\documentclass{amsart}
\usepackage[dvips]{color}
\usepackage{amssymb}
\usepackage[active]{srcltx}
\usepackage{amsmath, amssymb, euscript,  mathrsfs, enumerate, amsthm, amsfonts, graphicx, color}
\usepackage{calligra}   

\newcommand{\D}{\nabla}
\newcommand{\na}{ {\nabla}  }
\newcommand{\La}{ {\Delta}  }
\newcommand{\pp}{{\partial}}

\newcommand{\ds}{\displaystyle}
\newcommand{\al}{\alpha}

\newcommand{\C}{\mathcal \zeta}
\newcommand{\CC}{\mathcal C}

\newcommand{\R}{{\mathbb R}}

\newcommand{\e}{\epsilon}
\newcommand{\uu}{{\bar u}}
\newcommand{\uuu}{{\hat  u}}
\newcommand{\FF}{\bar{F}}

\newcommand{\be}{\begin{equation}}
\newcommand{\bee}{\begin{equation*}}
\newcommand{\ee}{\end{equation}}
\newcommand{\eee}{\end{equation*}}
\newcommand{\bs}{\begin{split}}
\newcommand{\esp}{\end{split}}

\newcommand{\la}{\lambda}
\newcommand{\F}{\mathcal  F}
\newcommand{\om}{{\omega}}

\newcommand{\Fn}{{\langle F, \nu \rangle}}

\newcommand{\Pn}{{\langle P, \nu \rangle}}

\newcommand{\xei}{{\langle F, \eei \rangle}}

\newcommand{\omn}{{\langle \om, \nu \rangle}}
\newcommand{\Fom}{{\langle F, \omega \rangle}}

\newcommand{\gah}{{\hat \gamma}}
\newcommand{\alh}{{\hat \alpha}}
\newcommand{\bw}{{\hat w}}
\newcommand{\brw}{{\bar w}}
\newcommand{\br}{{\bar r}}

\newcommand{\eei}{{\bf{e_i}}}
\newcommand{\eej}{{\bf{e_j}}}
\newcommand{\Fhn}{{\langle \hat F, \nu \rangle}}
\newcommand{\Fhne}{{\langle \hat F_\e, \nu \rangle}}

\newcommand{\EE}{\big( \frac{\pp}{\pp t} - \frac 1{H^2} \La \big )}

\newcommand{\bx}{{\bf x}}
\newcommand{\bbx}{{\bar x}}

\newtheorem{thm}{Theorem}[section]

\newtheorem{lem}[thm]{Lemma}
\newtheorem{prop}[thm]{Proposition}
\newtheorem{claim}{Claim}[section]

\theoremstyle{remark}
\newtheorem{rem}{Remark}[section]
\numberwithin{equation}{section}

\theoremstyle{definition}

\newcommand{\bremark}{\begin{remark} \em}
\newcommand{\eremark}{\end{remark} }

\title{Inverse Mean Curvature Evolution of entire graphs}

\author{Panagiota Daskalopoulos}
\address{ {\bf P. Daskalopoulos:} Department of Mathematics, Columbia University,  New York, NY 10027, USA.}
\email{pdaskalo@math.columbia.edu}
\author{Gerhard Huisken}
\address{ {\bf G. Huisken:} Department of Mathematics, University of T\"ubingen, T\"ubingen, Germany }
\email{gerhard.huisken@aei.mpg.de}

\begin{document} 

\begin{abstract}
We study the evolution  of strictly mean-convex entire graphs over $\R^n$ by {\em  Inverse Mean Curvature flow}.  First we  establish the {\em global existence} of starshaped entire graphs with {\em superlinear growth} at infinity. The main result in this work concerns the critical  case of {\em asymptotically conical}  entire convex graphs. In this case we show that there exists a time $ T < +\infty$,  which depends on the growth at infinity of the initial
data, such that the unique solution  of the flow exists for all $t < T$. Moreover, as $t \to T$ the solution converges to a  flat plane. 
Our techniques exploit the {\em ultra-fast diffusion } character of the fully-nonlinear flow, a property that implies that  the asymptotic behavior at spatial  infinity of our solution plays a crucial influence on the maximal time of existence, as such behavior propagates {\em infinitely fast}  towards the interior. 

\end{abstract}

\maketitle 

\tableofcontents

\section{Introduction} 

We consider a family of  immersions 
$F_t : M^n \to \R^{n+1}$ of n-dimensional mean convex hypersurfaces in $R^{n+1}$. We say that $M_t:=F_t(M^n)$ moves by
{\em inverse mean curvature flow}  if 
\be\label{eqn-IMCF} 
\frac{\partial}{\partial t} F(z,t)  =  H^{-1}(z,t)  \, \nu(z,t), \qquad z \in M^n
\ee
where $H(z, t) > 0$ and $\nu(z, t)$  denote  the mean curvature and exterior  unit normal of the surface
$M_t$ at  the point $F(z,t)$.

The compact case is well understood. It was shown by Gerhardt \cite{Ge} that for smooth compact star-shaped initial data of strictly positive mean curvature, 
the inverse mean curvature flow  admits  a smooth solution for all times which approaches a homothetically expanding spherical solution as $t \to +\infty$, see also Urbas \cite{Ur}.
For non-starshaped initial data it is well known that singularities may develop; in the case n = 2 Smoczyk \cite{Sm} proved that such singularities can only occur if the speed becomes unbounded, or, equivalently, when the mean curvature tends to zero somewhere during the evolution.
In \cite{HI1, HI2}, Huisken and Ilmanen  developed a new level set approach to weak solutions of the flow, allowing ``jumps'' of the surfaces and solutions of weakly positive mean curvature. Weak solutions of the flow can be used to derive energy estimates in General Relativity, see \cite{HI2}  and the references therein. 

In \cite{HI}, Huisken and Ilmanen studied further regularity properties 
of  inverse mean curvature flow with compact starshaped initial data of nonnegative mean curvature by a more classical approach
than their works in \cite{HI1, HI2}. 
They showed that starshapedness combined with the ultra fast-diffusion character of the equation, imply 
that at time $t >0$ the mean curvature  of the surface becomes strictly positive yielding to $C^\infty$ regularity.
No extra  regularity assumptions on the initial data are necessary. 
This work is reminiscent of well known estimates for the fast-diffusion equation 
\be\label{eqn-fd}
\varphi_t = \D_i (\varphi^{m-1}  \D_i \varphi), \qquad \mbox{on} \,\,\, \Omega \times (0,T)
\ee
on a domain $\Omega \subset \R^n$ and with exponents $m <1$. We will actually see in the next section that under inverse 
mean curvature flow, the mean curvature $H$ 
satisfies an  ultra-fast diffusion equation modeled on \eqref{eqn-fd} with $m=-1$. The work in \cite{HI} heavily uses that 
the initial surface is compact which corresponds to the domain $\Omega$ in \eqref{eqn-fd} being bounded. 
However, the case of non-compact initial data has never been studied before. 

Motivated by the  theory for  the Cauchy problem for the ultra-fast diffusion equation \eqref{eqn-fd} on $\R^n \times (0,T)$, 
we will study in this work  equation \eqref{eqn-IMCF} in the case that  the initial surface $M_0$ is  an entire graph over $\R^n$,   
i.e. there exists a vector $\omega \in \R^{n+1}$,
$|\omega| = 1$ such that
$$\langle \omega,  \nu \rangle  <      0, \qquad \mbox{on} \,\,\, M_0.$$
We will  take from now on $\omega$ to be  the direction of the $x_{n+1}$ axis, namely  $\omega = e_{n+1} \in 
\R^{n+1}$. 
A solution  $M_t$ of \eqref{eqn-IMCF} can then be   expressed  (at each instant  $t$) as  the graph  $\FF (x,t)= (x,\uu(x,t))$ of a function 
$\uu: \R^n \times [0,T) \to \R$. In this parametrization,  the inverse mean curvature flow \eqref{eqn-IMCF} is,  up to diffeomorphisms,  equivalent to  
\be\label{eqn-IMCF2} 
\left (\frac{\partial}{\partial t} \FF(x,t) \right )^\perp=  H^{-1}(x,t)  \, \nu(x,t), \qquad x \in \R^n
\ee
where $\perp$ denotes the normal component of the vector. Equation  \eqref{eqn-IMCF2} can then be expressed in terms of the height function  
$x_{n+1}=\uu (x,t)$ 
as the  fully nonlinear equation 
\be\label{eqn-u0}
\uu_t = - \sqrt{1+|D\uu|^2}\left ( \mbox{div}  \left ( \frac {D\uu}{\sqrt{1+ |D\uu|^2}} \right )   \right )^{-1}. 
\ee

\medskip

Entire graph solutions of the mean curvature flow have been studied  by Ecker and Huisken in \cite{EH1} (see also in \cite{EH}).
It follows from these works   which  are based on local apriori estimates,  that the mean curvature flow behaves in
some sense better than the  heat equation on $\R^n$: for an initial data $M_0$ which is an  entire graph over $\R^n$, no growth conditions
are necessary to  guarantee the long time existence of the flow for all times $t \in (0,+\infty)$. Entire graph solutions of fully-nonlinear flows by powers of Gaussian curvature were recently studied by Choi, Daskalopoulos, Kim and Lee in \cite{CDKL}. 
This is an example of slow diffusion which becomes degenerate at spatial infinity. Entire graph solutions of other fully-nonlinear flows
which are homogeneous of degree one were studied in \cite{CD}. 

\medskip 

This work concerns with the  long time existence of inverse mean curator flow for 
an initial data  $M_0$ which is an entire graph. In a first step we will establish  in theorem \ref{thm-ste0} the existence for all times $t \in (0, +\infty)$  of solutions with strictly meanconvex initial data $M_0 =\{x_{n+1} = \uu_0(x)\}$  having  superlinear growth at infinity, namely $\lim_{|x| \to +\infty} |D \uu_0(x)| = +\infty$, and satisfying a uniform "$\delta$-star-shaped" condition $\langle F - \bar x_0, \nu \rangle \, H \geq \delta >0$ for some 
$\bar x_0 \in \R^{n+1}. $ These  conditions for example hold for initial data $\bar u_0(x) = |x|^q$, for $q >1$.  
The proof of this result  uses in a crucial way  the evolution of 
$\langle F - \bar x_0, \nu \rangle \, H$ and the maximum principle which guarantees   that this quantity remains bounded from below at all times.

The main result of the paper proves long-time existence and uniform finite time singular convergence for convex entire graphs with {\em conical } behavior at infinity. 
We will assume that $M_0$ lies between two  rotationally symmetric cones $x_{n+1}=\C_i(\cdot,0)$, 
with $\C_1(\cdot,0):= \alpha_0 \, |x|$  and   $\C_2(\cdot,0):= \alpha_0 \, |x| + \kappa$, $x \in   \R^n$,  namely 
$\uu$ satisfies 
\be\label{eqn-cone0}
\alpha_0\,  |x|  \leq   \uu(\cdot, 0) \leq \alpha_0\,  |x| + \kappa, \qquad \mbox{on} \,\, \R^n
\ee
for some constants $\alpha_0 >0$ and $\kappa >0$. 
\be\label{picture1}{\mbox  \qquad  \includegraphics[width=2.5in]{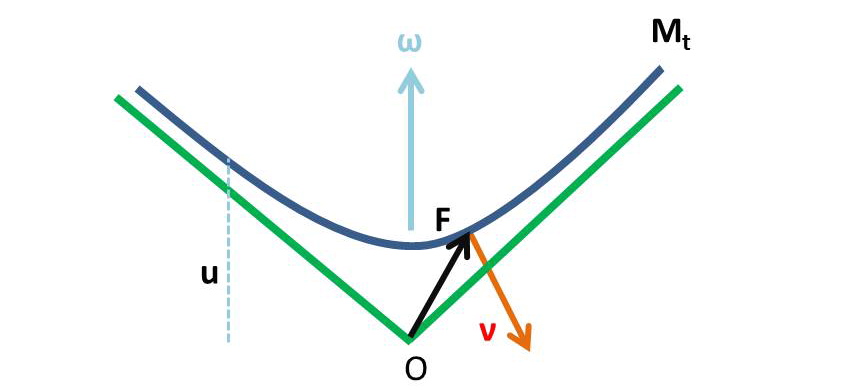} }\ee
We will see that  $M_t$ will remain convex and will lie between the
cones  $\C_1(x,t)=\alpha(t)\, |x|$ and $\C_2(x,t)=\alpha(t)\, |x|+\kappa$ which 
are  explicit solutions of  \eqref{eqn-IMCF}, namely 
\be\label{eqn-conet}
\alpha(t)\,  |x|  \leq   \uu(\cdot, t) \leq  \alpha(t)\,  |x| + \kappa, \qquad \mbox{on} \,\, \R^n. 
\ee
The coefficient $\alpha(t)$ is determined in terms of $\alpha_0$ by the ordinary differential equation   \eqref{eqn-ode1}. 
We will see in section \ref{sec-cone} that $\alpha(t) \equiv 0$ at a finite time $T=T(\alpha_0)$, which means that
the cone solutions $\C_i$ become flat at time $T(\alpha_0)$. Our goal in this work is to establish that the solution $M_t$
of \eqref{eqn-IMCF} with initial data $M_0$ will also exist up to this {\em critical time} $T$,  as stated next.

\begin{thm}\label{thm-main} Let  $M_0$ be  an entire convex $C^2$ graph $x_{n+1}=\uu_0(x)$ over $\R^n$
which lies between the two cones  as in  condition \eqref{eqn-cone0}.  
Assume in addition that  the mean curvature $H$ of $M_0$  satisfies the global bound 
 \be\label{eqn-conHF} 
c_0 \leq H \, \Fom \leq C_0.
 \ee
Let $T=T(\alpha_0)$ denote the lifetime of the cone with initial slope $\alpha_0$.  Then, there exists a unique $C^\infty$ smooth  solution $M_t$ of the \eqref{eqn-IMCF} for $t \in (0,T)$ which is an entire convex graph 
$x_{n+1}=\uu(x,t)$ over $\R^n$ and satisfies estimate  \eqref{eqn-conet} and has $H >0$ for all  $t \in (0,T)$.  As $t \to T$, the solution converges in $C^{1, \alpha}$ to some horizontal plane
of height $h \in [0, \kappa]$. 

\end{thm} 

%

\begin{rem} (Vanishing mean curvature)  
Condition \eqref{eqn-conHF}  implies that  the  initial data $M_0$ has strictly  positive mean curvature 
$H>0$. Actually for generic initial data which is a graph $x_{n+1}  = \bar u_0(x)$  satisfying \eqref{eqn-cone0} 
one expects that  $H(x,\bar u_0(x)) \sim |x|^{-1}$ as $|x| \to +\infty$. Hence,  under the extra assumption  $H >0$ 
one has that \eqref{eqn-conHF} holds.  It would be interesting to see if 
it is possible that the result of Theorem \ref{thm-main} is valid  under the weaker assumption that \eqref{eqn-conHF} 
holds only near spatial infinity, thus allowing  the  mean curvature $H$ to vanish on a compact set of $M_0$. 

\end{rem}

%
%

\begin{rem} The solutions in Theorem \ref{thm-main} have linear growth at infinity and they are  critical  in  the sense that  all other  solutions are expected to live longer. For conical at infinity initial data, one has  $\langle F - \bar x_0, \nu \rangle \, H \sim |x|^{-1}$ as $x \to +\infty$.  We will see that maximum principle arguments do not apply in this case to give us the required bound from below on $H$ which will guarantee existence.  One needs to use integral bounds involving $H$. Since our solutions are non-compact special account needs to be given to the behavior at infinity of our solution. This is one of the challenges in this work.   \end{rem}

\begin{rem}[Graphical parametrization] 
While we use  the graphical parametrization in conditions \eqref{eqn-cone0} and  \eqref{eqn-conet} and to establish the 
short time existence of our solution,   for all the a priori estimates,  which will occupy the majority  of this work,   we will use the 
geometric parametrization in \eqref{eqn-IMCF} where 
${\partial} F(z,t) / \partial t $ is assumed to  be in the direction of the normal $\nu$. 
This is because the evolution of the various geometric quantities becomes more simplified in the case of equation \eqref{eqn-IMCF}. In particular, $\bar u:= \bar u(x,t), x \in \R^n $  will denote the  height  function in the graph parametrization, while $u:= \Fom$ will denote the height   function in the geometric parametrization.
\end{rem}

\section{The geometric equations and preliminaries}\label{sec-prelim}

We recall the evolution equations for various geometric quantities
under the inverse mean curvature flow. Let $g = \{g_{ij} \}_{ 1 \leq i,j \leq n}$  and 
$A = \{ h_{ij} \}_{ 1 \leq i,j \leq n}$  be the first and second fundamental form of the evolving
surfaces, let $H = g^{ij}\, h_{ij}$   be the mean curvature, 
$\langle F - \bar x_0, \nu \rangle$ be  the support function with respect to a point $\bar x_0 \in \R^{n+1}$ and $d\mu $  the induced measure on $M_t$.

\begin{lem}[Huisken, Ilmanen \cite{HI}] \label{lem-HI1} Smooth solutions of \eqref{eqn-IMCF} with $H >0$ satisfy  
 \begin{enumerate}

{\em \item $ \ds \frac{\partial}{\partial t} g_{ij} =  \frac 2H \, h_{ij}$
\item $\ds \frac{\partial}{\partial t} d\mu = d\mu$
\item $\ds \frac{\partial}{\partial t}  \nu =  - \nabla H^{-1} =   \frac 1{H^2} \, \nabla H$
\item $\ds \frac{\partial}{\partial t} h_{ij} = 
 \frac 1{H^2}  \Delta h_{ij} -  \frac 2{H^3} \nabla_i H \nabla_j H + \frac{ |A|^2}{H^2}  \, h_{ij}$
\item $\ds \frac{\partial}{\partial t} H  =  \nabla_i \big  ( \frac 1{H^2}  \, \nabla_i H \big ) -     \frac{|A|^2}{H} = 
\frac 1{H^2} \Delta H  -  \frac 2{H^3} |\nabla H|^2 -     \frac{|A|^2}{H}$
\item $\ds \frac{\partial}{\partial t} H^{-1}    =  \frac 1{H^2} \Delta H^{-1}   +    \frac{|A|^2}{H^2}\, H^{-1}$
\item $\ds \frac{\partial}{\partial t} \langle F - \bar x_0, \nu \rangle  = \frac 1{H^2} \,  \Delta  \, \langle F - \bar x_0, \nu \rangle + \frac  { |A|^2}{H^2}  \, \langle F - \bar x_0, \nu \rangle$. }

\end{enumerate}
\end{lem}

We will next assume that $M_t$ is a graph in the direction of the vector $\omega$ and a smooth solution
of \eqref{eqn-IMCF} on $0 < t \leq \tau$ with $H >0$
and we  will derive the evolution of other useful geometric quantities under the IMCF.  
We will use the following identities
that hold in terms of a local orthonormal frame $\{ \eei \}_{1 \leq i \leq n}$ on $M_t$:
\be\label{eqn-lof}
 \D_{\eei} \nu = h_{ij} \, \eej, \quad \D_{\eei} \eej  = - h_{ij} \, \nu, \quad  \D_{\eei} \eei = - H\, \nu.
 \ee

\begin{lem} The norm of the position vector $|F|^2$ satisfies
\be\label{eqn-evolF}
\EE \, |F|^2 = - 2n H^{-2}  + 4 H^{-1} \,  \langle F, \nu \rangle.
\ee

\end{lem}

\begin{proof}
We have
$$\nabla_i |F|^2  = 2  \xei, \qquad 
\Delta |F|^2  =  2 n     - 2H \Fn$$
and
$$\frac{\partial }{\partial t}  |F|^2=2 \  \langle F, F_t \rangle  = 2\, H^{-1} \,   \Fn$$
which readily gives \eqref{eqn-evolF}. 

\end{proof}

\begin{lem}\label{lem-fnu} For any $\bar x_0 \in \R^{n+1}$  the support function
$H \,   \langle F- \bar x_0,   \nu \rangle $  satisfies the equation 
\be\label{eqn-evolfn}
\EE \big ( H \,   \langle F- \bar x_0,   \nu \rangle \big )  =  - \frac 2{H^3} \nabla H \cdot  \nabla 
\big ( H \,   \langle F- \bar x_0,   \nu \rangle \big ).
\ee
\end{lem}

\begin{proof} Readily follows by combining the evolution equations of  $H$ and  $  \langle F- \bar x_0$.
\end{proof}

\begin{lem}\label{lem-omnu} For a graph solution $M_t$, the quantity $\langle \omega, \nu \rangle$ satisfies 
\be\label{eqn-omnu}
\EE \omn =   \frac 
{ |A|^2}{H^2}  \, \langle \omega, \nu \rangle.
\ee

\end{lem}

\begin{proof}
We have
$$\frac{\partial}{\partial t}  \langle \omega, \nu \rangle =  \langle \omega, \frac{\partial}{\partial t}  \nu \rangle
= \frac 1{H^2} \,  \langle \omega, \nabla H \rangle.$$
On the other hand
$$\frac 1{H^2} \,  \La  \, \langle \omega, \nu \rangle = \frac 1{H^2} \,  \nabla_i \big  ( h_{ik}  \, \langle \omega, {\bf e_k} \rangle \big )=\frac 1{H^2}  \langle \omega, \na H\rangle  - \frac{ |A|^2}{H^2}  \, \langle w, \nu \rangle.$$
Hence, \eqref{eqn-omnu} holds.

\end{proof}

\begin{lem}\label{lem-Homnu} For a graph solution $M_t$,  the function  $\varphi := -H  \omn >0$ satisfies 
\be \label{eqn-ultrafast2}
\EE   \varphi  = - \frac 2{H^3} \nabla H \cdot  \nabla \varphi.
\ee

\end{lem} 

\begin{proof} Readily follows by combining the evolution of  $H$  and \eqref{eqn-omnu}. 
\end{proof}

\begin{lem}\label{lem-uuu} For a graph solution $M_t$, the height function $u:= \Fom$ satisfies the evolution equation
\be\label{eqn-uuu}
\EE u = \frac 2{H} \, \omn.
\ee
\end{lem}

\begin{proof}
It follows from 
$$\frac{\pp}{\pp t} \Fom = \frac 1{H} \, \omn \qquad \mbox{and} \qquad 
 \La \Fom = \D_i \langle \eei, \om \rangle = - H \, \omn.$$
\end{proof}

We next consider the quantity 
$$\Fhn:=  - \Fom \, \omn$$
which will play a crucial role in this work. 
We will  assume that our origin $0 \in \R^{n+1}$ is chosen so that $\Fom >0$ (in particular this holds if $M_t$ lies
above the cone $x_{n+1}= \al(t) \,|x|$ as in the picture \eqref{picture1}).  Since $\omn <0$, we have  $\Fhn > 0$ on $M_t$ for all $0 \leq t <\tau$. 

\begin{lem}\label{lem-fhatnu}
The  quantity $\Fhn:= - \Fom \, \omn >0$ satisfies the equation
\be\label{eqn-fhatnu}
\EE \Fhn =  \frac{|A|^2}{ H^2} \,\Fhn -   \frac 2{H} \, \omn^2+ \frac {h_{ij}}{H^2}\,  \langle \eei, \om \rangle \, \langle \eej, \om \rangle.
\ee
\end{lem}
\begin{proof}
Using  the evolution equations for  $\Fom$ and $\omn$ shown in 
lemmas   \ref{lem-omnu} and \ref{lem-uuu} respectively, we    conclude that 
$$\EE \Fhn =  \frac{|A|^2}{ H^2} \,\Fhn -   \frac 2{H} \, \omn^2+ \frac 1{H^2} \D_i \Fom \D_i \omn.$$
Since
$$\D_i \Fom \D_i \omn = h_{ij} \, \langle \eei, \om \rangle \, \langle \eej, \om \rangle$$
the above readily yields \eqref{eqn-fhatnu}.

\end{proof}

\begin{lem}\label{lem-fhatnuH} The product $v:= \Fhn \, H$ satisfies the evolution equation 
\be\label{eqn-fhatnuH}
 \frac{\pp}{\pp t} v - \D_i \big ( \frac 1{H^2} \D_i v\big ) =   -  2\,  \omn^2+ 
\frac {h_{ij}}{H}\,  \langle \eei, \om \rangle \, \langle \eej, \om \rangle. 
\ee
\end{lem}
\begin{proof}
Combining the evolution equation of $H$  given in Lemma \ref{lem-HI1} 
with \eqref{eqn-fhatnu}, gives
\begin{equation*}
\begin{split}
\EE v &= - \frac 2{H^2} \nabla \Fhn \nabla H -  \frac {2}{H^3} \,  \Fhn\, |\nabla H|^2 -  2\,  \omn^2+ \frac {h_{ij}}{H}\,  \langle \eei, \om \rangle \, \langle \eej, \om \rangle\\
& = - \frac 2{H^3} \nabla v \nabla H -  2\,  \omn^2+ \frac {h_{ij}}{H}\,  \langle \eei, \om \rangle \, \langle \eej, \om \rangle
\end{split}
\end{equation*}
from which \eqref{eqn-fhatnuH} readily follows. 
\end{proof}

\begin{lem}\label{cor-vinverse1}
Under the additional assumption that $M_t$ is convex, 
the function $v^{-1}:= (\Fhn \, H)^{-1}$ satisfies 
\be\label{eqn-vinverse1}
\frac{\pp}{\pp t} v^{-1} - \D_i \big ( \frac 1{H^2} \D_i v^{-1} \big ) \leq   -  \frac 2{H^2 v^{-1}}  |\nabla v^{-1}|^2 +
 2\,  \omn^2 v^{-2}.  
\ee

\end{lem}

\begin{proof}
Let $v:=\Fhn \, H$ as in Lemma \ref{lem-fhatnuH}.  We have 
\begin{equation*}
\begin{split}
\EE v^{-1} &= - \frac 1{v^2} \EE v  - \frac 2{H^2 v^3}  |\nabla v|^2   \\
&=  - \frac 2{H^3} \nabla v^{-1} \nabla H  -  \frac 2{H^2 v^3}  |\nabla v|^2 +
 2\,  \omn^2 v^{-2} - \frac {h_{ij}}{H}\,  \langle \eei, \om \rangle \, \langle \eej, \om \rangle  \, v^{-2}
\end{split}
\end{equation*}
which implies  \eqref{eqn-vinverse1} since $h_{ij} \,  \langle \eei, \om \rangle \, \langle \eej, \om \rangle\geq 0$ by
convexity. 

\end{proof}

Throughout this paper we will make use of the comparison principle in our non-compact setting. Although  rather standard under our assumptions,  for the convenience of the reader we will show next a proposition  which justifies this. The assumptions  are made so that it is  applicable in our setting. 

\begin{prop}[Comparison principle]\label{prop-comp}
Assume that $f \in C^2 (\R^n \times (0,\tau)) \cap C^0 (\R^n \times (0,\tau))$ satisfies the linear
parabolic inequality
$$f_t \leq a_{ij} \, D_{ij} f + b_i \, D_i f + c\, f, \qquad \mbox{on} \,\,  \R^n \times (0,\tau)$$
for some $\tau >0$ with coefficients which are measurable functions and satisfy the bounds 
\be\label{eqn-coeff1}
 \lambda |\xi|^2 \leq a_{ij}(x,t)\,   \xi_i \xi_j  \leq \Lambda |\xi|^2 \, (|x|^2+1), \qquad (x,t) \in \R^n \times [0,\tau], \,\, \xi \in \R^n 
 \ee
and
\be\label{eqn-coeff2}
 |b_i (x,t)| \leq \Lambda\, (|x|^2+1)^{1/2}, \qquad |c(x,t) | \leq \Lambda, \qquad (x,t) \in \R^n \times [0,\tau]
\ee
for some  constants $ 0 < \lambda <  \Lambda < +\infty$. Assume in addition that the solution $f$ satisfies the polynomial 
growth upper bound
$$f(x,t) \leq C\, (|x|^2+1)^p, \qquad \mbox{on} \,\,  \R^n \times [0,\tau]$$
for some $p >0$. 
If $f(\cdot,0) \leq  0$ on $\R^n$, then $f \leq 0$ on $\R^n \times [0,\tau]$.

\end{prop}

\begin{proof} To justify the application of the maximum principle it is sufficient to construct an appropriate  supersolution $\varphi$ of
our equation. We look for such a supersolution in the form
$$\varphi(x,t) = e^{\theta t} \, (|x|^2+1)^q$$
for some exponent $q >p$ and   $\theta=\theta (\Lambda, q) >0$ to be determined in the sequel. Defining the operator 
$$L \varphi := a_{ij} \, D_{ij} \varphi  + b_i \, D_i \varphi  + c\, \varphi$$
a direct calculation shows that under the assumptions on our coefficients we have 
$$\varphi_t - L \varphi  \geq  \big ( \theta  - C(\Lambda,q) \big ) \, \varphi$$
for some constant $C(\Lambda,q)$ depending only on $\Lambda, q$ and the dimension $n$. 
Hence, by choosing $\theta := 2 C(\Lambda,q)$ we conclude that $\varphi$ satisfies the inequality 
$$\varphi_t - L \varphi  > 0.$$
Now, setting $\varphi_\epsilon := \epsilon \, \varphi$, we have 
$$ f(x,t) \leq \varphi_\epsilon(x,t), \qquad \mbox{for} \,\, |x| \geq R_\epsilon, \,\, 0 \leq t \leq  \tau$$
for $R_\e >>1$, since we have taken $q >p$. Since, $f_\e \leq 0  \leq \varphi_\e$ by assumption, the maximum principle  guarantees that 
$f \leq \varphi_\epsilon$ on $\R^n \times (0,\tau)$ and by letting $\epsilon \to 0$
we conclude that $f \leq 0$ on on $\R^n \times [0,\tau]$ as stated in our proposition.
\end{proof}

We  will establish next, using the maximum principle,   local and global $L^\infty$ bounds  
from above  on the mean curvature of our solution $M_t$. Since our solution is convex these imply local and global bounds
on the second fundamental form.   We begin with the local bound. 
For the fixed  point $\bbx_0 \in \R^{n+1}$ and number $r >1$, we  consider the cut off function  
$$\eta:=  (r^2 - |F-\bbx_0|^2)_+^2.$$

\begin{prop}[Local bound from above on H] \label{prop-Hloc-above}
For a solution $M_t$ of \eqref{eqn-IMCF} on $t \in [0, \tau]$, $\tau >0$, if  $ \sup_{M_0} \, \eta (F(\cdot,0)) \,   H(\cdot,0) \leq C_0$, then 
\be\label{eqn-bHab}
\sup_{M_t} \, \eta (F(\cdot,t)) \,   H(\cdot,t) \leq  \max (C_0, 2n \, r^3).
\ee
\end{prop}
\begin{proof} 
We work on a local orthonormal frame $\{ \eei \}_{1 \leq i \leq n}$ on $M_t$ where identities \eqref{eqn-lof} hold. We   have
$$\nabla_i \eta = - 4 \, (r^2 - |F-\bbx_0|^2)_+ \, \langle F-\bbx_0, \eei \rangle = - 4 \eta^{1/2} \, \langle F-\bbx_0, \eei \rangle $$
and
$$\Delta \eta =  8 |(F-\bbx_0)^T|^2 - 4n \,  \eta^{1/2} + 4 \eta^{1/2}  \langle F-\bbx_0, \nu \rangle \, H$$
and
$$\frac{\partial \eta  }{\partial t} = - 4  \eta^{1/2} \,  \langle F-\bbx_0, F_t \rangle  
= - 4  \eta^{1/2} \,  \langle F-\bbx_0, \frac 1H \, \nu  \rangle = - 4  \eta^{1/2} \, \frac 1H \,  \langle F-\bbx_0, \nu \rangle.$$
We recall the $H$ evolves by the equation 
$$ \frac{\partial}{\partial t} H  =  \frac 1{H^2} \Delta H  -  \frac 2{H^3} |\nabla H|^2 -     \frac{|A|^2}{H}.$$
Using also  the  bound $|A|^2/H \geq H/n$, it follows that 
$$
\frac{\partial (\eta H ) }{\partial t}   \leq \frac 1{H^2} \eta \, \Delta H     -  \frac 2{H^3}  \,  |\nabla H|^2 \, \eta  -  
 \frac Hn \eta  -  4  \eta^{1/2} \langle F-\bbx_0, \nu \rangle.
$$
Since
$$\frac 1{H^2}  \, \Delta (\eta H) = \frac 1{H^2} \eta \, \Delta H + \frac 2{H^2} \,   \nabla_i H  \, \nabla_i \eta
+ \frac 1H\,  \Delta \eta$$
the above yields  
$$
\frac{\partial (\eta H ) }{\partial t}   \leq \frac 1{H^2} \, \Delta (\eta H)  - \frac 2{H^2} \,   \nabla_i H  \, \nabla_i \eta - \frac 8{H}\,  |(F-\bbx_0)^T|^2 +   \frac {4n}H \eta^{1/2}   -   \frac 2{H^3}  \,  |\nabla H|^2  \eta -  \frac Hn \eta.
$$
Using
$$ \frac 2{H^3}  \nabla_i H  \nabla_i (\eta H) = \frac 2{H^2} \,   \nabla_i H  \, \nabla_i \eta + 
\frac 2{H^3} \, |\nabla H|^2 \, \eta$$
we  conclude that   $\varphi:= \eta H$ satisfies
$$
\frac{\partial \varphi  }{\partial t}   \leq \frac 1{H^2} \, \Delta \varphi -  \frac 2{H^3}  \nabla_i H  \nabla_i \varphi
+   \frac {4n} \varphi  \, \eta^{3/2}   -  \frac \varphi n.
$$
For the  fixed $r >0$ and $\bbx_0 \in \R^{n+1}$, let 
$$m(t):= \max_{M_t } H \eta = \max_{M_t }  \varphi.$$
Since $\eta \leq r^4$, it  follows from the above differential inequality that $m(t)$ will decrease if 
$$ r^6 \, \frac {4n } {m(t)}  -  \frac  {m(t)} n \leq 0  \iff m^2(t) \geq 
4n^2 \, r^6  \iff m(t) \geq  2n \, r^3.
$$
Hence 
$$
m(t) \leq \max \, (\, m(0), 2n\, r^3). 
$$

\end{proof}

\begin{rem} We note that Proposition \ref{prop-Hloc-above} does not require the convexity of $M_t$. 
\end{rem}

\begin{prop}[Global bound from above on H]\label{prop-Habove} For a convex graph solution $M_t$
of \eqref{eqn-IMCF} on $t \in [0,\tau]$, if  $ \sup_{M_0}  \Fom H(\cdot,0) < \infty$, then 
\be\label{eqn-Habove}
\sup_{t \in [0,\tau]} \sup_{M_t}  \Fom H  \leq \sup_{M_0}  \Fom H.
\ee
\end{prop}
\begin{proof} We will compute the evolution of $\Fom H \geq 0$ from the evolution of $H$ given in Lemma  \ref{lem-HI1} and 
the evolution of the height function $\Fom$ given by \eqref{eqn-uuu}. Indeed, combining these two equations leads
$$\EE (\Fom H)  =   -  \frac 2{H^3} |\nabla H|^2 \, \Fom - \frac 2{H^2}  \nabla H \cdot \nabla \Fom  -   \frac{|A|^2}{H} \, \Fom + 2\, \omn.$$
Writing
$$   \frac 2{H^3} |\nabla H|^2 \, \Fom + \frac 2{H^2}  \nabla H \cdot \nabla \Fom = \frac 2{H^3}  \nabla H \cdot \nabla (\Fom H)$$
and using $\Fom \geq 0$ and   $\omn \leq 0$, we conclude that $ \Fom H$ satisfies
$$\EE \big (  \Fom H  \big ) \leq  - \frac 2{H^3}  \nabla H \cdot \nabla \big ( \Fom H \big ) $$
and  the bound \eqref{eqn-Habove} readily follows by the comparison principle.   
\end{proof}

\section{Self-similar  solutions}\label{sec-cone}

%
%

We will study in this section self-similar entire graph solutions $x_{n+1}=\uu(x,t)$ of the IMCF equation \eqref{eqn-u0} which have
polynomial growth at infinity,  namely $\uu (x,t) \sim |x|^q$, with $q \geq 1$. These solutions are all rotationally symmetric. 

First,  we consider rotationally  symmetric infinite cones in the direction of the vector $\omega=e_{n+1}$. 
If the vertex $P \in \R^{n+1}$ of the cone is the origin $0 \in \R^{n+1}$ for its  position vector $F$, then $\Fn=0$, 
otherwise $\Fn=\Pn$.  Since the cone is a rotationally symmetric graph,  in its graph parametrization  
$\FF(x,t):=(x, \C(|x|,t))$, $x \in \R^n$ it is given by
a height function 
\be\label{eqn-cone1}
\C(r,t) = \al(t)\, r + \kappa, \qquad  r:=|x|, \qquad x \in \R^n
\ee
for a constant $\kappa \in \R$.   The function $\C$ is a solution of the equation 
\be\label{eqn-IMCF-radial}
\C_t = -\frac{(1+\C_r^2)^2}{ \C_{rr} + (n-1) \,  (1+\C_r^2)\, \C_r /r}
\ee
which is satisfied by any rotationally symmetric graph $\FF(x,t):=(x, \uu(|x|,t))$, $x \in \R^n$ which evolves by equation \eqref{eqn-IMCF2}. 
\bee\label{picture2}{\mbox  \qquad  \includegraphics[width=2.5in]{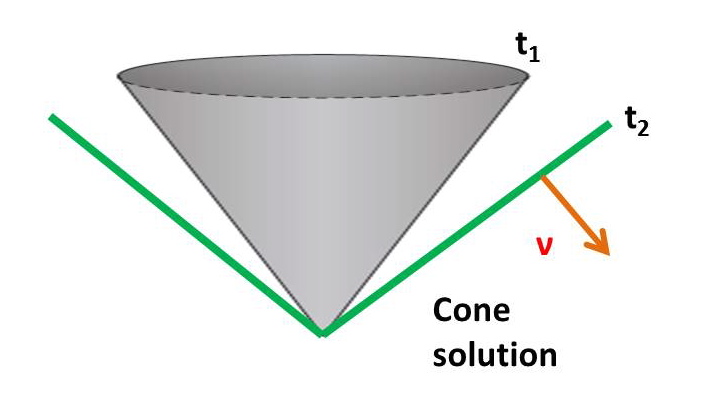} }\eee
It follows from \eqref{eqn-IMCF-radial} that $\alpha(t)$ satisfies 
the ODE
\be\label{eqn-ode1}
\al'(t) = - \frac 1{n-1}  \left (\alpha(t) + \frac 1{ \alpha(t)} \right ).
\ee
On the conical solution we have  
\be\label{eqn-con1}
\omn = -\frac 1{\sqrt{1+\al(t)^2}}
\ee
and 
\be\label{eqn-Hcone}
H(r,t) = \frac{(n-1)\, \al(t)}{r\, \sqrt{1+\al(t)^2}}.
\ee
We conclude that on the conical solution
$$v:= - \Fom \, \omn\, H = \gamma(t):= \frac{(n-1)\, \al(t)^2}{1+\al(t)^2}.$$
Setting 
$$\beta(t):= \omn^2=\frac 1{1+\al(t)^2}$$ we have 
$$\gamma(t)= (n-1)\, (1- \beta(t)).$$ 
To compute the evolution of $\gamma(t)$ and $\beta(t)$ it is simpler to use the equations  
 \eqref{eqn-omnu} and \eqref{eqn-fhatnu}  which directly give
$$\beta'(t) = \frac 2{(n-1)} \beta(t) \quad \mbox{and} \quad  \gamma'(t) = - 2 \, \beta(t).$$
This is correct since $\gamma(t)$ and $\beta(t)$ are independent 
of the parametrization.  
We conclude that
\be\label{eqn-ode4}
\beta'(t) = \frac 2{(n-1)} \beta(t) \quad \mbox{and} \quad \gamma'(t) = 2\,  \big ( \frac {\gamma(t)}{n-1} - 1 \big ).
\ee
Solving the last ODE's    with initial conditions  $\beta_0 \in (0,1)$ and $\gamma_0:= (n-1)\, (1-\beta_0)$ yields
\be\label{eqn-gammat}
\beta(t) = \beta_0 \, e^{2t/(n-1)}  \quad \mbox{and} \quad  \gamma(t) = (n-1) \big (1 - (1-\frac{\gamma_0}{(n-1)} )\, e^{2t/(n-1)} \big ).
\ee
Finally, recalling that $1+\al^2(t)=\beta(t)^{-1}$, we conclude that the {\em slope} $\al(t)$ of the conical solution is 
$$\al(t) = \big ( (1+\al_0^2) \, e^{-2t/(n-1)} - 1 \big )^{1/2}.$$
It is clear from the above equations that the conical solution will become flat at time $T(\alpha_0)$ given by
\be\label{eqn-T}
T(\alpha_0)= \frac {n-1}2 \log (1+\al_0^2) = \frac {n-1}2 \log \big ( \frac{n-1}{(n-1)-\gamma_0} \big ). 
\ee

\medskip

Next, let us briefly discuss other self-similar solutions of equation \eqref{eqn-u0} which exists for all time $t >0$ and they  are also rotationally
symmetric.   It is simple  to observe that the time $t$ cannot be scaled in the fully-nonlinear
equation \eqref{eqn-u0}. Nevertheless,  equation \eqref{eqn-u0} admits (non-standard) self-similar solutions of the form
\be
\label{eqn-ss1}
\uu_\lambda(x,t) = e^{\la t} \, \uu_\lambda( e^{-\la t} \, x), \qquad (x,t)  \in \R^n \times \R
\ee
for a suitable range of  exponents  $\lambda >0$, where $x_{n+1}=\uu_\la(x)$ are entire convex graphs over $\R^n$.
The function $x_{n+1}=\uu_\la(x)$, $x \in \R^n$ satisfies  the fully-nonlinear elliptic
equation
\be\label{eqn-ula}
 \mbox{div}  \left ( \frac {D \uu}{\sqrt{1+ |D \uu|^2}} \right )  = \frac 1\la \, \frac{ \sqrt{1+|D u|^2}}{x \cdot D \uu - \uu}.  
\ee
Although equation \eqref{eqn-ula} may possess non-radial solutions,   restricting   ourselves  to rotationally 
symmetric solutions  $x_{n+1}=\uu_\la(r)$, $r=|x|$ , it follows that $\uu:= \uu_\la(r)$ satisfies the ODE
\be\label{eqn-ur}
\uu_{rr} + (n-1) \,  (1+u_r^2)\, \frac{\uu_r }r - \frac 1{\la} \,   \frac{((1+\uu_r^2)^2}{r\, \uu_r -\uu}=0.
\ee  
One needs to impose condition $\uu(0)= \kappa<0$ 
to guarantee the   existence of an entire convex solution. The following was shown by the authors and J. King in \cite{DH}. 

\begin{thm}[The existence of self-similar solutions] \label{thm-sss} For every  $\la > 1/(n-1)$ and $\kappa <0$, there exists a unique rotationally symmetric entire convex solution $x_{n+1} = \uu_\la (r)$ of 
\eqref{eqn-ula} on $\R^n$ with $\uu_\la(0) =\kappa$. In addition, $\uu:=\uu_\la$ satisfies the following flux condition at $r=+\infty$
\be\label{eqn-finfty2} \lim_{r \to \infty} \frac{r\, u_r(r)}{u(r)}=q, \qquad q:=\frac{\la\, (n-1)}{(n-1)\, \la -1}.
\ee
\end{thm}

The condition \eqref{eqn-finfty2} shows that $u_\la(x) \sim |x|^\frac{\la\, (n-1)}{(n-1)\, \la -1}$ as $|x| \to \infty$. Notice that since $\la$ is
any number  $\la > 1/(n-1)$, the exponent $q:=\frac{\la\, (n-1)}{(n-1)\, \la -1}$ covers the whole range 
$q  \in (1,+\infty)$, hence each solution $u_\la$ has a polynomial growth at infinity larger than that of the conical solution 
$x_{n+1} = \al(t)\, |x| + \kappa$.  It would be interesting to see whether the limit
$\lim_{\la \to +\infty} u_\la$ gives  the conical solution  or possibly another solution with super linear behavior as $|x| \to \infty$.


\section{The super-linear case and short time existence}\label{sec-ste}

In this section we assume that $M_0$ is an entire  graph $\{x_{n+1} = \uu(\cdot,0)\}$  over $\R^n$ 
in the direction of the vector $\om:= e_{n+1}$. We first prove  long-time existence for solutions to \eqref{eqn-IMCF} with superlinear growth that are {\it $\delta$-starshaped} (see below) and then  we establish the short-time existence for the critical case of convex solutions that lie between the rotationally symmetric cones $x_{n+1}=\C_i(\cdot,0)$, with $\C_1(\cdot,0):= \alpha_0 \, |x|$  and   $\C_2(\cdot,0):= \alpha_0 \, |x| + \kappa$, $x \in   \R^n$, as in \eqref{eqn-cone0}. The solutions $M_t$ are then given as the graph of $\uu(\cdot,t)$ satisfying  \eqref{eqn-u0}.

\smallskip 

In our first result, Theorem \ref{thm-ste0} below, we will assume that the   initial data $M_0$ is an 
entire graph $x_{n+1} = \uu_0(x)$  over $\R^n$ with superlinear growth at infinity,  i.e.
\begin{equation}\label{superlinear}
|D \uu_0(x)| \to \infty,  \,\,\,   \uu_0(x) \to \infty,  \qquad \text{for} \, \,|x| \to \infty
\end{equation} 
and is strictly starshaped with a uniformity condition:

We say that $M_0$ is {\em $\delta$-starshaped}  if there is a point $\bar x_0 \in \R^{n+1} \cap \{(x, x_{n+1})| u_0(x) <x_{n+1} \}$ and 
a constant $\delta>0$ such that the mean curvature $H$ satisfies 
\begin{equation}\label{starshaped}
H \langle F -\bar x_0, \nu \rangle \geq \delta >0
\end{equation}
everywhere on $M_0$.  By Lemma \ref{lem-fnu}, this condition which provides  a scaling invariant quantitative 
measure for the starshapedness of a hypersurface,  is preserved under inverse mean curvature flow whenever
the maximum principle can be applied.

\begin{thm}[Existence for superlinear initial data]\label{thm-ste0} Assume that  the initial surface 
$M_0$ is an entire graph $\{(x, x_{n+1})|\,x \in \R^n, x_{n+1}=\uu_0(x)\}$ with 
$\uu_0\in C^2(\R^n)$ and satisfying the assumptions (\ref{superlinear}) and (\ref{starshaped}), for some $\bar x_0 \in \R^{n+1}$. 
Then,  there is a smooth solution $F: M^n \times [0, \infty) \to \R^{n+1}$ of the inverse mean curvature flow 
\eqref{eqn-IMCF} for all times $t>0$ that can be written as a graph $M_t = F(\cdot, t)(M^n) =\{x_{n+1}=\uu(x,t)\} $ with initial data $M_0$. If $\uu_0$ is convex, then the  solution $M_t$ is also convex for all time. 

\end{thm}

\begin{rem} (i) It is easy to see that the assumptions (\ref{superlinear}) and (\ref{starshaped}) are satisfied if 
$\uu_0(x) = |x|^q$,   provided $q > 1$.\hfill\break
(ii) The condition ``$\delta$-starshaped" is reminiscent but different from the ``$\delta$-non-collapsed" condition 
that has been used in mean curvature flow.
\end{rem}

\begin{rem}\label{rem-local} In the case of  convex initial data,  the condition  
$\uu_0 \in C^2(\R^n)$  in  Theorem \ref{thm-ste0}  may be replaced by the weaker condition $\uu_0 \in C^2_{\text{loc}}(\R^n)$,  since the mean curvature is uniformly controlled on compact sets. 

\end{rem}
\begin{proof} By translating the surface we may assume that $\bar x_0 = 0$ is the origin of $\R^{n+1}$; then
$\uu_0(x) \geq C_0$ is bounded below everywhere by some negative constant. For the proof we will assume that $\uu_0\in C^{2, \alpha} (\R^n)$.
For initial data just in $C^2$ as our theorem states,  the result   follows by approximation in view of the estimates we establish.

By the assumption (\ref{superlinear}) we may choose $R_0 > 1$ such that $|D\uu_0(x)| \geq 100$
provided $|x| \geq R_0$. We want to approximate $M_0$ with compact surfaces by replacing the
region $\{ \uu_0 \geq R\}$ of the surface with the reflection of the region $\{ \uu_0 \leq R\}$ on the plane at height $R$.
Set,  for each $R\geq R_0$
\begin{equation}
\hat u_{0,R}(x) := 2R - \uu_0(x)
\end{equation} 
and set
\begin{equation}
E_{0,R} := \{(x, x_{n+1}) \in \R^n \times \R :  \, \,  \uu_0(x) < x_{n+1}  < \hat u_{0,R} (x)\}, \quad \Sigma_{0,R} := \partial E_{0,R}.
\end{equation} 
The outer unit normals $\nu$ and $\hat \nu$ to the lower and and upper part of $\Sigma_{0,R}$
are given by
\begin{equation}
\frac{(D\uu_0, -1)}{\sqrt{1 + |D\uu_0|^2}}\quad \text{and} \quad \frac{(-D\hat  u_{0,R}, 1)}{\sqrt{1 + |D\hat u_{0,R}|^2}} 
\end{equation}
respectively, where by definition $D \hat  u_{0,R} (x) = - D \uu_0(x)$ and the mean curvatures $H$ and $\hat H$
satisfy $\hat H(x, \hat u_{0,R}(x)) = H(x, \uu_0(x))$. In particular we get that for $(x,x_{n+1}) \in \Sigma_{0,R}$ 
with $x_{n+1} = R$
 \begin{equation}\label{Lipschitz}
 \langle \nu, \hat \nu\rangle (x,x_{n+1}) = \frac{|D\uu_0(x)|^2 -1}{|D\uu_0(x)|^2 +1} \geq 1 - 10^{-3}.
\end{equation}
In addition  we compute that $\hat F_{R}(x,0) := (x, \hat u_{0,R}(x)$ satisfies
\be\label{R-starshaped}
\begin{split}
\hat H (x, &\hat u_{0,R}(x))  \langle \hat F_R(x,0),  \hat \nu\rangle(x, \hat  u_{0,R}(x)) \nonumber \\ 
&= \frac{H(x, \uu_{0,R} (x)) }{\sqrt{1+|D\uu_{0,R}(x)|^2}} \Big( \langle (x, \uu_{0,R}(x), (-D\uu_{0,R} (x), 1)\rangle\Big) \nonumber \\
&= \frac{H(x, \uu_0(x)) }{\sqrt{1+|D\uu_0(x)|^2}} \Big( \langle (x, 2R - \uu_0(x), (D\uu_0(x), 1)\rangle\Big)\nonumber \\
&= H(x, \uu_0(x)) \langle F(x), \nu\rangle(x, \uu_0(x))
+ \frac{2R}{\sqrt{1+|D\uu_0(x)|^2}} \geq \delta>0
\end{split}
\ee
such that the surface $\Sigma_R = \partial E_R$ is again uniformly $\delta$-starshaped. If the initial function
$u$ is convex, all regions $E_R$ are convex as well.

Next we smoothen out the region $x_{n+1}= \uu_0(x) = R$ using mean curvature flow:

\begin{lem}\label{MCF} Given $\uu_0 \in C^{2, \alpha}(\R^n)$, 
for each $\Sigma_R = \partial E_R$ as above there is a one-parameter family
of hypersurfaces $\Phi: S^n \times [0, s_R] \to \R^{n+1},\, \Phi(\cdot, s) (S^n) = \Sigma_R(s),\, s_R>0,$ with initial data 
$\Sigma_R(0)  = \Sigma_R$ satisfying mean curvature flow
\begin{equation}\label{mcf}
\frac{d}{ds} \Phi(p,s) =  \overrightarrow H (\Phi(p,s)), \qquad p\in S^n, \,  s \in [0, s_R].
\end{equation}
The surfaces $\Sigma_R(s),\, s\in [0, s_R],$ are smooth and approach $\Sigma_R$ in $C^{0, 1/2}$ as $s \to 0$. 
For small $s_R >0$ they are $\tilde\delta$-starshaped with $\tilde\delta \geq \delta - o(s^{\alpha/2})$. 
We may choose $\sigma_R \in (0, s_R]$ such that $\sigma_R \to 0$ as $R \to \infty$ and all 
$\Sigma_R(\sigma_R)$ are uniformly bounded in $C^{2, \alpha}$. If $u_0$ is convex, then $\Sigma_R(\sigma_R)$ is strictly convex with some lower bound $\lambda_R>0$ for all its principal curvatures.
\end{lem}

\begin{proof} [Proof of Lemma \ref{MCF}] $\Sigma_R$ is a uniformly Lipschitz hypersurface over its tangent spaces in view of (\ref{Lipschitz}), 
so we may solve mean curvature flow for a short time with $\Sigma_R$ as initial data, compare (\cite{EH}, Theorems 3.4 and 4.2),  
to obtain a smooth solution $\Sigma_R(s)$ for (\ref{mcf}) on some time interval $(0, s_R], s_R>0$ 
which approaches the initial data in $C^{0, 1/2}$ as $s \to 0$. For small $s>0$ this solution has strictly positive mean curvature;
this follows from the fact that  $\Sigma_R$ provides a barrier for Mean curvature flow and can be approximated by a smooth 
surface of strictly positive mean curvature from the inside, e.g. by gluing in arbitrary small sectors of an approximate
cylinder along the edge $x_{n+1} = R$. Since $\bar u_0\in C^{2, \alpha}(\R^n)$, the interior regularity estimates in
\cite{EH} combined with Schauder theory yield
\begin{equation}\label{Hoelder}
\Big\vert H(\Phi(p,s)) - H(\Phi(p,0))\Big\vert \leq c(R )\,  r(p)^{-1-\alpha} \, s^{{\alpha}/{2}}
\end{equation}
where $r(p ) = |\Phi^{n+1}(p,0) -R|$ is the distance to the singular set $\{x_{n+1} = R\}$ and $c(R )<\infty$ depends 
on the  $C^{2, \alpha}$- norm of $\uu_0$ on $B_R(0)$. If $y = \Phi(p_0,0)$ with $y^{n+1} = R$ is a point on the
edge formed between ${\rm graph}(\uu_0)$ and ${\rm graph}(\hat u_{0,R})$, then in view of the uniform Lipschitz
estimates for small $s$ a rescaling of $\Sigma_R(s)$ around $y$ for $s\to 0$ converges 
to the solution of mean curvature flow $\Gamma(s) \times \,\R^{n-1}, 0<s<\infty$, where $\Gamma(s)$ is the unique selfsimilar 
expanding solution of curve-shortening flow in the 2-plane containing $e_{n+1}$ and $D \uu_0(p_0 )$
that is associated with the angle between $\nu(p_0 )$ and $\hat \nu(p_0 )$. The unit normal to this solution 
interpolates between $\nu(p_0 )$ and $\hat\nu(p_0 )$ while its geodesic curvature decays exponentially, in fact it has been 
shown in (\cite{Hal}, Lemma 6.4 ) that its geodesic curvature $\kappa(r,s)$ at time $s$ and distance $r$ from $y$ is given by
\begin{equation}\label{expcurve}
\kappa(r,s) = \frac{1}{\sqrt{2s}} \, \kappa_{\max}(\frac{1}{2}) \,  \exp\Big(\kappa_{\max}^2(\frac{1}{2})\, s -\frac{r^2(p,s)}{4s}\Big).
\end{equation}
Here $\kappa_{\max}({1}/{2})$ is determined by the opening angle between $\nu, \hat \nu$ in such a way that
$\kappa_{\max}({1}/{2}) \to 0$ as this opening angle tends to $0$, or, equivalently, $|D\uu_0| \to \infty$ on the edge $\{x_{n+1} = R\}$ as $R\to \infty$. Let $\kappa_R$ be the largest such $\kappa_{max}(1/2)$ arising from an opening angle on the edge $\{x_{n+1} =R\}$. If we then choose $\sigma_R \in (0, S_R]$ smaller than $\kappa^2_R$ we see that the surfaces $\Sigma_R(\sigma_R)$ are uniformly bounded in $C^{2, \alpha}$ in view of 
(\ref{expcurve}) and Schauder theory while approximating $M_0$ uniformly in $C^{2, \alpha}$ as $R \to \infty$ since $\sigma_R \leq \kappa^2_R \to 0$. Combining then the
$\delta$-starshapedness in (\ref{starshaped}), (\ref{R-starshaped}) with the estimates (\ref{Hoelder}) and (\ref{expcurve})
we see that $\Sigma_R(\sigma_R)$ must be $\tilde\delta$-starshaped with $\tilde\delta \geq \delta - o(\sigma_R^{\alpha/2})$.
If the function $\uu_0$ is convex then $\Sigma_R(s), s>0$ will be uniformly convex by the strong 
parabolic maximum principle, i.e. there will be $\lambda_R>0$ such that the eigenvalues
$\lambda_i, 1\leq i\leq n$ of the second fundamental form all satisfy $\lambda_i \geq \lambda_R$ everywhere
on $\Sigma_R(\sigma_R)$.\end{proof}

\noindent {\it Proof of Theorem \ref{thm-ste0} continued}. Given a sequence of radii $R_i \to \infty$ we may choose parameters $\sigma_i := \sigma_{R_i} \to 0$ as
in the preceding  lemma with corresponding smooth approximating surfaces $\Sigma^i := \Sigma_{R_i}(\sigma_i)$ with
$\Sigma^i = \partial E^i$ such that $E^i \subset E^j$ for $i<j$ and $\Sigma^i$ is $\delta_i$-starshaped with
$\delta_i \to \delta$ as $i\to \infty.$
From the work of Gerhardt  \cite{Ge} (see also Urbas \cite{Ur}), for each approximating surface $\Sigma^i$ 
there is a smooth solution $\Sigma^i(t), t\in [0, \infty)$ of inverse mean curvature flow starting from $\Sigma^i$
that approaches a homothetically expanding sphere as $t\to \infty$. 
We now combine the $\delta_i$- starshapedness for each $R>0$ with the local bound on the mean curvature obtained in proposition \ref{prop-Hloc-above} such that
\begin{equation}\label{localstarshaped1}
0<\delta_i\leq H \langle F, \nu \rangle \leq  \max (4 \max_{M_0 \cap B_R(0)} H, \, \frac{8n}{R}) 
\langle F, \nu \rangle := C_1(R) \langle F, \nu \rangle
\end{equation}
and therefore
\begin{equation}\label{localstarshaped2}
0<\frac{\delta_i}{C_1(R)} \leq \langle F, \nu \rangle \leq  R 
\end{equation}
holds everywhere on $\Sigma^i(t) \cap B_{{R}/{2}}(0)$ when $R_i > 2R$. Hence $\Sigma^i(t)$ is uniformly
starshaped in $B_{{R}/{2}}(0)$ and we may use the local curvature bound in (\cite{Hei}, Theorem 3.6 and Remark 3.7) to conclude that 
\begin{equation}\label{localcurvest}
|A|^2 \leq  C_2 \max ( \max_{M_0 \cap B_R(0)} |A|^2, R^{-1} \max_{M_0 \cap B_R(0)} H + R^{-2}).
\end{equation}
holds everywhere on $\Sigma^i(t) \cap B_{{R}/{4}}(0)$ when $R_i > 2R$, where $C_2$ depends on $n$ and $(R \, \max_{M_0 \cap B_R(0)} H) $.
Thus the solutions satisfy uniform curvature estimates independent of $i$ on  any compact set. 
Higher regularity then follows from known theory, see in  \cite{K}. To obtain a subsolution we choose for each $T<\infty$ an $0<\alpha_0=\alpha_0(T) <\infty, \kappa(T) > -\infty$ such that the conical solutions 
$$
\zeta(x, t) = \alpha(t)\,  |x| + \kappa(T), \quad \alpha(0) = \alpha_0
$$
from section \ref{sec-cone}  provide a lower barrier for all  $\Sigma^i(t)$ on  $ t\in [0, T)$. Here $\alpha_0$ is chosen so that $T$ is the lifetime of the cones $\zeta(\cdot, t)$. It follows that we can pass  to the limit to obtain a solution 
$M_t$ of the inverse mean curvature flow which is defined for all $t \in (0,+\infty)$ and is  $C^\infty$ smooth. 
Note that $M_t$ is again an entire graph: For each $\rho >0$ the initial hypersurface $M_0$ is $\delta$-starshaped also with respect to $\bar x_{\rho} = \rho\om$ since $H \langle F -\bar x_{\rho}, \nu \rangle
= H \langle F, \nu \rangle - \rho H \langle \om, \nu \rangle \geq \delta >0$ as $\langle \om, \nu \rangle \leq 0$.
If $R_i >\rho$ this will also be true for $\Sigma_i(0)$ and hence, by the maximum principle, on all $\Sigma_i(t)$.
Thus $\langle F -\rho\om, \nu \rangle >0$ for all $\rho >0$ everywhere on all $M_t$. Dividing by $\rho$ and letting $\rho\to\infty$ on compact subsets yields $\langle -\om, \nu \rangle \geq 0$ and hence $\langle -\om, \nu \rangle > 0$ by the strong maximum principle as desired.

If the initial surface $M_0$ is convex, then each $\Sigma_i(0)$ is uniformly convex by Lemma \ref{MCF}. Then
in view of the result of  Urbas \cite{Ur}  the surfaces $\Sigma^i(t), 
t\in [0, \infty)$  are also uniformly convex proving that all limit surfaces $M_t$ are convex in this case.
This completes the proof of the longtime existence of solutions with superlinear, $\delta$-starshaped initial data, as stated in Theorem \ref{thm-ste0}. 

\end{proof}

We will give next a short time existence result for convex initial data  $M_0$ which lies between two cones as
in condition \eqref{eqn-cone0}.

\begin{thm}[Short time existence of asymptotically  conical solutions] \label{thm-ste} Let  $M_0$ be  an entire convex 
 graph $x_{n+1}=\uu_0(x)$ over $\R^n$ which satisfies   condition   \eqref{eqn-cone0}.  
Assume in addition that  the mean curvature $H$ of $M_0$ satisfies  the global bounds  
 \be\label{eqn-steHt0}
 0 < c_0 \leq  H \,  \Fom \leq C_0.
\ee
Then, there exists $\tau >0$ and a unique $C^\infty$ smooth  solution $M_t$ of  \eqref{eqn-IMCF} for  
$t \in (0,\tau]$ 
which is an entire convex graph 
$x_{n+1}=\uu(x,t)$ over $\R^n$ and satisfies condition \eqref{eqn-conet}. Moreover,  on $M_t$ we have 
 \be\label{eqn-steHt}
c_\tau \leq H \, \Fom   \leq C, \qquad \mbox{for all} \,\,  t \in (0,\tau]
\ee 
for a constant $c_\tau >0$ depending on $\tau$ and $C:= \max\,  ( C_0, 2n)$. 
\end{thm}

\begin{rem} We note that on the graph $M_0$ of any convex function $\bar u_0 \in C^{2}_{\text{loc}}(\R^n)$ which satisfies
condition   \eqref{eqn-cone0} one has 
\be\label{eqn-fn00}
0 \leq - \Fn \leq C_0
\ee
 for some constant $C_0$ which can be taken 
without loss of generality to  be equal to $C_0$ in  \eqref{eqn-steHt0}. 
\end{rem}

\begin{proof} 

For $\e \in (0,1)$, consider the approximations $M^\e_0$ of the initial surface $M_0$  defined as entire graphs 
$x_{n+1} = \uu_{0,\e}(x)$, with 
\be\label{eqn-uee}
\uu_{0,\e}(x) = \uu_0(x) + \e \, (|x|^2 +1), \qquad x \in \R^n.
\ee
Then, each $\uu_{0,\e} $ satisfies the conditions of Theorem \ref{thm-ste0}  (with $\uu_{\e,0}  \in C^2_{\text{loc}}$ instead of $\uu_{0,\e} \in C^2$) and
in addition it is strictly convex. By Theorem \ref{thm-ste0}  and Remark \ref{rem-local} there exists a  solution 
 $M^\e_t$  to \eqref{eqn-IMCF} on $t \in (0,+\infty)$ with initial data $M^\e_0$. 
In addition   $M^\e_t$ are smooth entire convex graphs given by $x_{n+1} = \uu_\e(x,t)$, $x \in \R^n$. 
The functions $\uu_\e$  satisfy equation \eqref{eqn-u0}. 
Since $\uu_{0,\e}$ satisfies $\alpha_0 \, |x| \leq \uu_{0,\e}(x) \leq u_{0,1}(x)$ for all $\e \in (0,1)$, 
the comparison principle implies that for any $0 < \e_1 < \e_2 <1$ we have 
$$\alpha(t)\, |x| \leq \uu_{\e_1}(x,t) \leq  \uu_{\e_2}(x,t) \leq \uu_1(x,t), \qquad (x,t) \in \R^n \times  [0,T)$$
where $T$ denotes the extinction time of $\alpha(t)$.  In particular, the
monotone limit
$\uu:= \lim_{\e \to 0} \uu_\e$
exists and satisfies 
\be\label{eqn-uue}
\alpha(t)\, |x| \leq  \uu(x,t) \leq \uu_1(x,t).
\ee
We will show next that $\uu$ is a solution of \eqref{eqn-u0} with initial data $\uu_0$. 

\begin{claim} \label{claim-1} There exists $\tau >0$ for which the limit  $\uu $ is a smooth  convex solution of 
\eqref{eqn-u0} on $\R^n \times (0,\tau)$ with  initial data $\uu_0$. 

\end{claim}

\begin{proof}[Proof of Claim \ref{claim-1}] Consider the approximations $\uu_\e$ and denote by $H_\e$ the mean curvature of
$M^\e_t$. Set $v_\e := \Fhne \, H_\e$, where  $ \Fhne$ denotes the quantity 
$\Fhne := - \langle F_\e , \omega \rangle \, \langle  \omega, \nu_\e  \rangle$
on  $M_t^\e$. Each $v_\e$ satisfies the equation 
\eqref{eqn-fhatnuH} and since each  $M_t^\e$ is convex the last term on the righthand side of \eqref{eqn-fhatnuH} is nonnegative.
Since  $\langle \omega, \nu_\e \rangle \leq 1$, we  conclude that each $v_\e$ satisfies
\be\label{eqn-fhatnuH22}
 \frac{\pp}{\pp t} v_\e- \D_i \big ( \frac 1{H^2} \D_i v_\e \big ) \geq   -  2. 
\ee
Moreover, our initial conditions on $\uu_0$ guarantee that $v_\e \geq c_0 >0$ for a uniform in $\e$ constant $c_0$. The differential inequality 
\eqref{eqn-fhatnuH22} implies that 
\be\label{eqn-stvep}
v_\e:= \Fhne \, H_\e  \geq c_0/2 >0, \qquad \mbox{on} \,\, M_t^\e, \,\, t \in [0,\tau]
\ee
if we choose $\tau:= c_0/4$.  On the other hand, our initial
assumption that $H  \, \Fom  \leq C_0$ on $M_0$   implies that   $H  \, |F| \leq C_1$ on $M_0$,  which 
in turn gives a uniform in $\e$ bound $H_\e   \leq C_2 /(1+|x|)$  on $M^\e_0$,  for a uniform in $\e$ constant $C_2$.  
Proposition \ref{prop-Hloc-above}, implies  the bound 
\be\label{eqn-baH100}
H_\e \leq C\, (1+|x|)^{-1}, \qquad \mbox{on}  \,\, M^\e_t,  \,\, t \in [0,\tau].
\ee 
Combining the two estimates   yields  
\be\label{eqn-He}
0 < c_R \leq H_\e (\cdot,t)   \leq  C, \qquad \mbox{on} \,\, |x| \leq R, \,\,  t \in [0,\tau], 
\ee  for  uniform in $\e$ and $t$  constants
$c_R, C$.  These bounds  guarantee that the equation \eqref{eqn-u0} is uniformly parabolic in $\epsilon$ on compact sets and by standard regularity arguments the limit $\uu$ 
will be a smooth convex solution of  \eqref{eqn-u0}  with initial data $\uu_0$. By passing to the limit in \eqref{eqn-stvep}
and \eqref{eqn-baH100} we conclude that 
\be\label{eqn-steH}
 \Fhn \, H   \geq c >0 \qquad \mbox{and} \quad H \leq C \, (1+|x|)^{-1},   \qquad \mbox{on} \,\, M_t, \,\, t \in [0,\tau]
\ee
for $c:=c_0/2$. 
\end{proof} 

It remains to show that the solution $\uu$ satisfies the upper bound $\uu(x,t) \leq \alpha(t) \, |x| + \kappa$ in \eqref{eqn-conet}.
To this end, we will first show that $\uu(\cdot,t)$ has linear growth at infinity which will allow us to apply the comparison
principle Proposition \ref{prop-comp}. 

\begin{claim}\label{claim-lba} The limit $\uu$ satisfies the linear bound $$\uu(x,t) \leq \beta \, |x| + \kappa_1$$
for some constants $\beta >0$ and $\kappa_1 >0$. 

\end{claim} 

\begin{proof}[Proof of Claim \ref{claim-lba}] 
First notice that for any pair $(\beta, \kappa)$ with $\beta>0, \kappa \in \R$ an elementary calculation shows that
the subgraph of the conical surface $\zeta (x) = \beta \, r + \kappa, \, r=|x|$, is equal to the complement
of a natural family of spheres lying above it:
\begin{equation}\label{equ-charcone}
\{(x,x_{n+1} ) \in \R^n \times \R   \, | \,  x_{n+1}  <  \beta |x| + \kappa\} =
\bigcap_{\rho>0, \tilde\kappa \geq\kappa}\Big( \R^{n+1}\setminus B_{\rho}(0, \rho\sqrt{1+\beta^2} +\tilde\kappa)\Big).
\end{equation}
Since $\uu_{0}(x) \leq \alpha_0 |x| + \kappa$, for each $\rho_0 >0$ and $\tilde\epsilon>0$ we may now choose 
$\epsilon_0 = \epsilon_0(\rho_0, \tilde\epsilon))>0$ 
such that for all $0<\epsilon <\epsilon_0, 0<\rho<\rho_0$ the balls
$B_{\rho}(0, \rho\sqrt{1+\alpha_0^2} +\tilde\kappa + \tilde\epsilon),\, \tilde\kappa \geq \kappa,$ are contained in the epigraph 
\begin{equation*}
\{(x,x_{n+1}) \in \R^n \times \R    \, | \, x_{n+1}  > \bar u_{0, \epsilon}(x)\}
\end{equation*}
of the approximating functions $\bar u_{0, \epsilon}$ given by \eqref{eqn-uee}. Using the barrier principle for 
IMCF applied to the resulting graphs of $\bar u_{t, \epsilon}$ and the balls expanding by IMCF,
$B_{\rho(t)}(0, \rho\sqrt{1+\alpha_0^2} +\tilde\kappa + \tilde\epsilon), \, \rho(t) = \rho \exp(t/n)$ 
we conclude from the monotone convergence of the $\bar u_{t, \epsilon}$  in the limit $\tilde\epsilon \to 0$ 
that the balls $B_{\rho(t)}(0, \rho\sqrt{1+\alpha_0^2} +\tilde\kappa), \, \tilde\kappa \geq \kappa$ 
are contained in the epigraph of $\bar u_{t} = \lim \bar u_{t, \epsilon}$ for each $\rho >0$. In other words,
\begin{equation*}\label{equ-baru1}
\{(x, x_{n+1}) \in \R^n \times \R \, | \, x_{n+1}  = \bar u(x,t)\} \subset
\bigcap_{\rho>0, \tilde\kappa \geq\kappa}\Big( \R^{n+1}\setminus B_{\rho(t)}(0, \rho\sqrt{1+\alpha_0^2} +\tilde\kappa)\Big).
\end{equation*}
Now note that
\begin{equation*}
\rho\sqrt{\alpha_0^2 +1} = \rho(t)\exp(-t/n)\sqrt{\alpha_0^2 +1} = \rho(t) \sqrt{\beta(t)^2 +1},
\end{equation*}
where 
\begin{equation*}
\beta(t) = \sqrt{\exp(\frac{-2t}{n})(1 + \alpha_0^2) -1}, \quad t\in [0, n\log\sqrt{1+\alpha_0^2}) \cap [0, T_{\max}(\bar u)).
\end{equation*}
Since $\cap_{\rho(t)} =\cap_{\rho>0}$ we get
\begin{equation*}\label{equ-baru2}
\{(x,x_{n+1} ) \in \R^n \times \R\,  | \, x_{n+1}  = \bar u(x,t)\} \subset
\bigcap_{\rho(t), \tilde\kappa \geq\kappa}\Big( \R^{n+1}\setminus B_{\rho(t)}(0, \rho(t)\sqrt{1+\beta^2(t)} +\tilde\kappa)\Big)
\end{equation*}
which implies the claim
\begin{equation*}
\bar u(x, t) \leq \beta(t) |x| + \kappa 
\end{equation*}
in view of (\ref{equ-charcone}) as required. Notice that $\beta(t) > \alpha(t)$ for $t>0$ such that this estimate cannot yet yield 
the optimal upper bound. \end{proof}

It remains to show that $\uu(x,t) \leq \alpha(t) \, |x| + \kappa$. This  simply follows from the next claim, by comparing  
$\uu$ with the conical solution $\C_2(x,t):= \alpha(t) \, |x| + \kappa$. 

\begin{claim}\label{claim-lcomp} Assume that $\uu_1, \uu_2$ are two smooth and convex entire graph solutions of equation
\eqref{eqn-u0} on $\R^n \times (0,\tau]$ for some $\tau >0$ which satisfy the bounds 
\be\label{eqn-linear}
 \alpha \, |x| \leq \uu_i(x,t) \leq \beta \, |x| + \kappa, \qquad i=1,2, \qquad \mbox{on}\,\, \R^n \times [0,\tau]
 \ee
for some constants $0 < \alpha \leq \beta $ and $\kappa >0$.  Assume in addition that $\uu_i$, $i=1,2$  both satisfy conditions 
\eqref{eqn-steH}.  If $\uu_1(\cdot,0) \leq \uu_2(\cdot,0)$, then  
$\uu_1(\cdot,t) \leq \uu_2(\cdot,t)$ on $\R^n \times (0,\tau]$. 

\end{claim} 

\begin{proof}[Proof of Claim \ref{claim-lcomp}] To simplify the notation for any function $\uu$ on $\R^n \times [0,\tau]$,  we set 
 \be\label{eqn-deff}
 \F(D^2\uu, D\uu):= -  \sqrt{1+|D\uu|^2} \left (  \mbox{div}  \left ( \frac {D\uu}{\sqrt{1+ |D\uu|^2}} \right ) \right )^{-1}
 = - \frac{\sqrt{1+|D\uu|^2} }{H}.
 \ee
Since both $\uu_1$ and $\uu_2$ satisfy
\begin{equation}\label{eqn-uu} 
\frac {\partial}{\partial t} \uu  =  \F(D^2\uu, D\uu)
\end{equation}
setting $\uu_s :=  s\, u_1 + (1-s) u_2$, we have 
\bee\begin{split}
\frac {\partial}{\partial t} (\uu_1 - \uu_2) &= \F(D^2\uu_1, D\uu_1) - \F(D^2\uu_2, D\uu_2) \\
&= \int_0^1 
\frac {d}{ds}  \F(D^2 \uu_s, D\uu_s)\, ds \\
& =  \int_0^1  \frac {\partial \F}{\partial \sigma_{ij}}  \, D_{ij}  (\uu_1 - \uu_2) +  \frac {\partial \F}{\partial p_i}  \, D_i (\uu_1 - \uu_2) \, ds\\
&= a_{ij} \, D_{ij}  (\uu_1 - \uu_2) + b_i \, D_i (\uu_1 - \uu_2)
\end{split} 
\eee
where 
$$a_{ij} :=  \int_0^1  \frac {\partial \F}{\partial \sigma_{ij}} (D^2 \uu_s, D\uu_s)\, ds, \qquad b_i := \int_0^1   \frac {\partial \F}{\partial p_i}  (D^2 \uu_s, D\uu_s)\, ds.
$$
The  uniqueness assertion of our theorem will directly follow from Proposition \ref{prop-comp}  if we show that the coefficients 
$a_{ij}$ and $b_i$ satisfy conditions \eqref{eqn-coeff1} and \eqref{eqn-coeff2}. To this end, we observe using \eqref{eqn-deff}
that
$$\frac {\partial \F}{\partial \sigma_{ij}} =  \frac{\sqrt{1+|D\uu|^2} }{H^2} \, \frac {\partial H}{\partial \sigma_{ij}}.$$
Since
$$H = \frac 1{\sqrt{1+|D\uu|^2}} \, \left (\delta_{ij} - \frac{D_i \uu D_j \uu}{1+|D\uu|^2} \right ) D_{ij} \uu$$
we conclude that
$$\frac {\partial \F}{\partial \sigma_{ij}} =  \frac 1{H^2 } \,  
\left (\delta_{ij} - \frac{D_i \uu D_j \uu}{1+|D\uu|^2} \right ).$$
Moreover, a direct calculation shows that
$$ \big | \frac {\partial \F}{\partial p_i} \big | \leq C \, |x|^2 \, \frac{ |D^2 \uu| }{(1+|D\uu|^2)^{1/2}}.$$
Observe next that  \eqref{eqn-steH}, \eqref{eqn-conet} and the convexity of $\uu$ imply the uniform
bound $H \geq c/(1+|x|)$ on $M_t$, $t \in [0,\tau]$ for some $c >0$ and we also have the  uniform  bound from above 
$H \leq C/(1+|x|)$. It is easy to conclude then that $a_{ij}$ satisfies 
$$ \lambda \,\xi^2  (1+|x|^2) \leq a_{ij}(x,t)  \, \xi_i \xi_j \leq  \Lambda \,\xi^2 \, (1+|x|^2)$$
for some positive constants  $\lambda, \Lambda$. 
Also, convexity implies the bound $|D^2 \uu| /(1+|D\uu|^2)^{1/2} \leq C\, H$ which in turn gives 
$$|b_i(x,t)  | \leq  C \, |x|^2 \, \frac{ |D^2 \uu| }{(1+|D\uu|^2)^{1/2}} \leq C \, |x|^2 \, H \leq C\, (1+|x|).$$
We can then apply  Proposition \ref{prop-comp} to conclude that $\uu_1 \leq \uu_2$, as claimed. 
\end{proof}

Let us now conclude the proof of the Theorem  \ref{thm-ste}.  We have  shown above  that the limit $\uu$ is a smooth convex 
solution of \eqref{eqn-u0} which satisfies conditions \eqref{eqn-conet} and  \eqref{eqn-steH}, which in turn imply 
\eqref{eqn-steHt}. The uniqueness assertion of the theorem readily follows by Claim \ref{claim-lcomp} since \eqref{eqn-conet}
and \eqref{eqn-steHt} also imply  \eqref{eqn-steH}. 

\end{proof} 

We will next compute the  behavior at infinity of  $v:= \Fhn \, H$, where $\Fhn:= - \Fom \, \omn$. 
This will be  crucial for  the proof of Theorem 
\ref{prop-Hbelow} which is the  main a'priori estimate in this work. 

\begin{prop}\label{prop-asb} Under the assumptions of Theorem \ref{thm-ste}, the function $v:= \Fhn \, H$ on $M_t$
satisfies  the asymptotic behavior 
\be\label{eqn-good}
\lim_{|F(z,t)| \to +\infty} v(F(z,t),t) = \gamma(t), \qquad \mbox{for all} \,\, t \in (0, \tau]
\ee
with $\gamma(t)$ given by \eqref{eqn-gammat}. 

\end{prop}   

\begin{proof}   We will use the graph representation $x_{n+1}=\uu(x,t)$, $(x,t) \in \R^n \times [0,T)$,  of the solution  $M_t$ of \eqref{eqn-IMCF} 
for $t \in (0,\tau]$, as given by Theorem \ref{thm-ste} and we will show that 
\be\label{eqn-good2}
\lim_{|x| \to +\infty} v(x, \bar u(x,t),t) = \gamma(t), \qquad \mbox{uniformly on}\,\, [\tau_0, \tau]
\ee
for all $\tau_0 \in (0, \tau/2)$, which readily yields \eqref{eqn-good}. 
The function  $\uu$ satisfies the equation \eqref{eqn-u0} and conditions \eqref{eqn-conet} and \eqref{eqn-steH}.  We may then consider
$H$ and $\Fhn$  as functions of $(x,t) \in \R^n \times  [0,\tau]$. Throughout the proof  $c, C$ will denote positive constants which may change from line to line but always remain  uniform in $t$,  for $t \in [0,\tau]$. 

Since  $\Fhn  = \uu/ (1+|D\uu|^2)^{1/2}$,  by   \eqref{eqn-conet} 
 we have $\Fhn \leq C \, |x|$, on $[0,\tau]$. Combining  this with   \eqref{eqn-steH}  yields
\be\label{eqn-Hgood2}
c\, (1+|x|)^{-1}  \leq H(x ,t)  \leq C\, (1+|x|)^{-1}, \qquad \mbox{on } \, \, \R^n \times  [0,\tau]. 
\ee
 In addition, \eqref{eqn-conet} and the convexity of $\uu$ imply the gradient estimate
\be\label{eqn-grad}
\sup_{  \R^n \times [0,\tau]} |D \uu (x,t)| \leq C.
\ee

The function $\uu$ satisfies  \eqref{eqn-uu},  where the fully nonlinear operator  $\F$ is given by \eqref{eqn-deff}. 
For this proposition we will use an a priori estimate for fully-nonlinear parabolic equations which was proven by G. Tian and X-J. Wang in \cite{TW}
(Theorem  1.1 in \cite{TW}) to show that the mean curvature $H$ of our surface remains sufficiently close to that 
of the cone
$\C_1(x,t) = \alpha (t)\, |x|$,  for $|x|$ sufficiently large,  because of condition \eqref{eqn-conet}. To this end, let 
$\tilde u$ be the function defined
by
\be\label{eqn-baru} 
\uu(x,t) = \alpha(t) \, |x| \, \big ( 1 + \tilde u(x,t) \big ). 
\ee

\smallskip
We notice that the ellipticity of the operator $\F$  depends on  $H(x,t) \sim |x|^{-1}$,  for
$|x|$ large.  Hence, equation \eqref{eqn-uu} becomes singular as $|x| \to \infty$. To avoid this issue it is more
convenient to work in cylindrical coordinates for $|x| \geq \rho$, with $\rho$  large.  Let $\tilde u$ 
be the function defined by \eqref{eqn-baru} and express $\tilde u$ in polar coordinates $\tilde u(r,\theta_1,\cdots, \theta_{n-1},t)$,
$r=|x|$.
We introduce cylindrical $(s,\theta_1,\cdots, \theta_{n-1})$ with 
$s:=\log r$   and define the function $\uuu(s,\psi ,t)$ in terms of $\tilde u$, setting
$${\tilde  u} (r,\psi,t)=  \uuu(s,\psi,t), \qquad \quad s:=\log r, \,\, \psi:= (\theta_1, \cdots, \theta_{n-1}) \in S^{n-1}.$$
It follows by a direct calculation that $\uuu$ satisfies an equation of the form 
$$\uuu_t = \hat \F(D^2 \uuu, D \uuu, \uuu)$$
where $D^2 \uuu := D^2_{s,\psi} \uuu $ and $D \uuu := (\uuu_s, D_\psi \uuu)$ denote first and second derivatives 
with respect to the cylindrical variable $(s,\psi)$. The nonlinearity $\hat \F$ also depends on $\gah$. 

Observe first that \eqref{eqn-conet} implies the bound
\be\label{eqn-hatu2} 
0 \leq \uuu (s,\psi, t) \leq \frac{\kappa}{\alpha(t)} \, e^{-s} \leq C\, e^{-s}, \qquad \mbox{on } \, \, \R \times S^{n-1} \times  [0,\tau]
\ee
with $C=C(\tau)$. 
Also, a  direct calculation shows that 
\be\label{eqn-fnu15}
\Fn = \frac {\alpha(t) \,  e^s\, \uuu_s}{\sqrt{1+ \alpha^2(t) (1+ \uuu+ \uuu_s)^2 + |D_\psi \uuu|^2)}}
\ee
where by  \eqref{eqn-grad}, $  \sqrt{1+ \alpha^2(t) (1+ \uuu + \uuu_s)^2 + |D_\psi \uuu|^2)}  = 1 + |D \uu|^2  \leq C$.  
On the other hand, the ondition $|\Fn | \leq C_0$ on $M_0$ (see in \eqref{eqn-fn00}) and the maximum principle on the evolution 
of $\Fn$ given in Lemma \ref{lem-HI1}  implies that $| \Fn| \leq C$ on $M_t$ for $t \in [0,\tau]$. Hence, \eqref{eqn-fnu15} implies the bound
\be\label{eqn-hatu5} 
0 \leq \uuu_s(\cdot,t)  \leq C  \, e^{-s},  \qquad \mbox{on } \, \, \R \times S^{n-1}  \times  [0,\tau]. 
\ee
For any $s_0 \geq 0$, and $\tau_0 \in (0,\tau)$, we  set  
$$Q_{s_0,\tau_0}^2:= [s_0-2, s_0+2] \times S^{n-1}  \times [\tau_0/2, \tau], \quad  Q_{s_0,\tau_0}^1:= [s_0-1, s_0+1]  \times S^{n-1}  \times [\tau_0, \tau]$$
so that $Q_{s_0,\tau_0}^1 \subset     Q_{s_0,\tau_0}^2$. 
It is not difficult to verify (using  \eqref{eqn-hatu2} and \eqref{eqn-hatu5}) that  the nonlinearity $\hat \F$ satisfies the assumptions  of Theorem 1.1 in \cite{TW} on $Q_{s_0,\tau_0}^2$, for any $s_0 \geq 0$
with bounds that are independent from $s_0$ (as long as $s_0 \geq 0$).  It follows that from Theorem 1.1 in \cite{TW}  that for any $s_0 \geq 0$
$$
\| D^2_{s,\psi} \uuu \|_{C^{\alpha,\alpha/2}(Q_{s_0,\tau_0}^1)} \leq C_{\tau_0}   \,\,  \| \uuu \|_{L^\infty( Q_{s_0,\tau_0}^2)}
$$
for an exponent $\alpha >0$.  
Here   $C_{\tau_0}$ depends on  the initial data, $\tau$ and $\tau_0$,   but is independent of $s_0$. This also implies the bound
\be\label{eqn-c2alpha5}
\| e^{s_0}  \uuu \|_{C^{\alpha,\alpha/2}(Q_{s_0,\tau_0)}^1}  \leq C_{\tau_0}   \,\,  \| e^{ s_0}  \uuu \|_{L^\infty(Q_{s_0,\tau_0}^2)}. 
\ee
Combining \eqref{eqn-c2alpha5} with the  bounds  \eqref{eqn-hatu2}  and \eqref{eqn-hatu5} gives
$
\| e^{s_0}\uuu \|_{C^{2+\alpha}_{\text{cyl}}(Q_{s_0,\tau_0}^1)} \leq C_{\tau_0} < \infty.
$
Since the constant $C_{ \tau_1} $ is independent of $s_0$ (as long as $s_0 \geq 0$) we finally obtain the bound
\be\label{eqn-c2alpha12}
\| e^{s}  \uuu \|_{C^{2+\alpha}_{\text{cyl}}(\CC \times [\tau_0, \tau])} \leq C_{ \tau_0} < \infty
\ee
where $\CC$ denotes the half cylinder given by $\CC:= [0,+\infty) \times S^{n-1}.$
This estimate shows that $\uuu$ is uniformly small in $C^{2+\alpha}_{\text{cyl}}$ norm as $s \to +\infty$. Expressing the mean curvature $v:=\Fhn \, H$ in
cylindrical coordinates we readily deduce that \eqref{eqn-good2} holds, which also implies \eqref{eqn-good}.

\end{proof}

\section{$L^p$ bounds on  $1/H$}\label{sec-Lp}

We will assume in this section   that $M_t$ is a  solution of \eqref{eqn-IMCF} on $[0,\tau]$
as given by  Theorem  \ref{thm-ste} which satisfies condition \eqref{eqn-conet}
and that  $\tau < T-3\delta$,  for some $\delta >0$ and small, where $T=T(\alpha_0)$  denotes the extinction time of $\alpha(t)$ given in terms of $\alpha_0$ by  \eqref{eqn-T}.  Recall that $v:=\Fhn \, H $, where $\Fhn= -\Fom \, \omn$. 
Our goal is  to establish  a'priori bounds on   suitably weighted  $L^p$ 
norms of $v^{-1}(\cdot,t)$ on $M_t$, for any $p  \geq 1$,  that depend on $\delta$ but are independent of $\tau$.  
In the next section we will use these $L^p$ bounds and a Moser iteration argument to bound the $L^\infty$ norm of $v^{-1}$
on $M_t$. This $L^\infty$ bound constitutes the main a priori estimate on which  the proof of the long time existence of the flow is based upon.  We begin with the following straightforward  observation which will be frequently used in the sequel. 

\begin{lem} Assume that  $M_t$  is  an entire convex graph over $\R^n$  satisfying \eqref{eqn-conet} with $\al(t) \leq \al_0$.  Then, 
$\Fhn:=  - \Fom \, \omn$, satisfies 
\be\label{eqn-FFF}
 \Fhn  \geq \lambda (n, \alpha_0)  \, \sqrt{\gamma(t)}\, |F|. 
\ee
\end{lem}
\begin{proof}
We begin by noticing that the lower bound in  condition  \eqref{eqn-conet} implies the bound 
$$\Fom  \geq \frac{\alpha(t)}{\sqrt{1+\alpha^2(t)}}\, |F|.$$
In addition, it follows from the convexity of $M_t$ and \eqref{eqn-conet} that 
${\displaystyle - \omn \geq \frac 1{\sqrt{1+\al(t)^2}}.}$
Thus,  
$$\Fhn:=- \Fom \, \omn  \geq \frac{\al(t)}{1+\al(t)^2} \, |F| \geq \la(\al_0,n) \, \sqrt{\gamma(t)}\, |F|.$$
The last inequality follows from the definition of $\gamma(t):= (n-1)\al(t)^2/(1+\al(t)^2)$ and 
$\al(t) \leq \al_0$. 
\end{proof}

We recall that $v^{-1}:= (\Fhn \, H)^{-1}$ satisfies  equation \eqref{eqn-vinverse1} and by   \eqref{eqn-good} 
$  \lim_{|F| \to \infty} v(z,t)= \gamma(t)$,   where $\gamma(t)$ satisfies the ODE \eqref{eqn-ode4} 
with initial condition $\gamma(0):=\gamma_0:=(n-1) \al_0^2/(1+\al_0^2)$.

Let $\gah(t)$ denote the solution of the ODE \eqref{eqn-ode4} with initial condition $\gah(0):=\gah_0:=(n-1) \alh_0^2/(1+\alh_0^2)$
for some number $\alh_0$ that satisfies $0 < \alh_0 < \al_0$.  Then, $\gah(t) < \gamma(t)$. 
Denote by $\hat T=\hat T(\hat \alpha_0)$ the vanishing time of $\gah$ that clearly satisfies $\hat T < T=T(\alpha_0)$. 
For a given number $\delta >0$ (small) we will choose  from now on $\alh_0$ such that the vanishing time $\hat T$ of  $\gah$ 
satisfies 
$$T-2\delta \leq \hat T \leq T-\delta.$$
For that choice of $\gah$ we will have  $\gah(t) < \gamma(t)$, for all $t < \hat T$. 
Hence, if we set 
\be\label{eqn-defw}
w(\cdot,t):=\gah(t)  \, v(\cdot,t)^{-1}=\gah(t)  \,(\Fhn \, H)^{-1}
\ee then by  \eqref{eqn-good} we have 
\be\label{eqn-limitw}
\lim_{|F(z,t)|\to \infty} w(z,t) =\gah(t)\, \gamma(t)^{-1} <1, \qquad t \in [0,\tau], \,\, \tau < \hat T . 
\ee

We will next compute the evolution of $w$ from the evolution of $v^{-1}$, shown in \eqref{eqn-vinverse1},  and 
the ODE for $\hat \gamma$,  shown in  \eqref{eqn-ode4}.
Indeed, if we multiply \eqref{eqn-vinverse1} by $\gah(t)$ and  use \eqref{eqn-ode4}, we obtain
\be\label{eqn-eqw2}
\frac{\pp}{\pp t} w  - \D_i \big ( \frac 1{H^2} \D_i  w \big ) \leq  -  \frac 2{H^2 w}  |\nabla w|^2  + 2\,  \omn^2 \gah^{-1}  w^2 +
c_1  \gah^{-1}  \, w
\ee
with
$$c_1:= \frac{2\gah(t)}{n-1} - 2 <0, \qquad \mbox{since} \,\,\, \gah < \gamma(t):= \frac{(n-1) \, \alpha^2(t)}{1+\alpha^2(t)}  \leq (n-1).$$
Next, we set
$$\bw:= (w-1)_+.$$
Because of \eqref{eqn-limitw}, 
for  each given $t \in [0,\tau]$, $\tau <\hat T <T$,  the function $\bw(\cdot,t)$ satisfies
\be\label{eqn-zero}
\bw(\cdot,t)\equiv 0, \qquad \mbox{for} \,\,  |F| \geq R(t)
\ee for some $R(t) < \infty$. Notice that the main difficulty in our proof comes from the fact that we don't know 
that $R(t)$ is uniform in $t$.

\begin{lem} [Energy Inequality] \label{lem-energy} Under the assumptions of Theorem  \ref{thm-ste}, for any $p \geq 0$, $q:=(p+3)/2$ 
the function $\bw:= (w-1)_+$ with $w:=\gah \, (\Fhn \, H)^{-1}$ satisfies  
\be\label{eqn-eqw3} 
\begin{split}
\frac{d}{dt} \int_{M_t} & \bw^{p+1} \, d\mu + 2 \lambda^2 \gah^{-2} \gamma   \int_{M_t} |F|^2  |\D \bw^q|^2 \, d\mu \leq \\
&  \leq  
 C(p)\,  \gah^{-1} \left (  \int_{M_t} \bw^{p+2} \, d\mu  
+   \int_{M_t} \bw^{p+1} \, d\mu +   \int_{M_t} c_0 (z,t) \,  \bw^p \, d\mu \right ) 
\end{split} 
\ee
with $\lambda, C(p)$ positive constants that depend  on the  initial data (and $C(p)$ also on linearly $p$) and 
\be\label{eqn-co}
c_0(z,t):=  2\, \big (\omn^2  +\frac{\gah}{n-1} - 1 \big )_+.
\ee

\end{lem}
\begin{proof}
If we first set $\bar w := w-1$, we see from \eqref{eqn-eqw2} that 
\bee\begin{split}
\frac{\pp}{\pp t} \bar w  - \D_i \big ( \frac 1{H^2} \D_i  \bar w \big ) & \leq  -  \frac 2{H^2 w}  |\nabla \bar  w|^2 +  2 \gah^{-1} \omn^2  (\bar w+1)^2
+ \gah^{-1} c_1 (\bar w +1)\\
&\leq -  \frac 2{H^2 w}  |\nabla \bar w|^2 + 2 \gah^{-1} \omn^2   \bar w^2 + \gah^{-1} (4\omn^2  + c_1) \bar w + \gah^{-1} (2\omn^2   + c_1 ).
\end{split}
\eee
We next  observe that since $\omn^2 \leq 1$ and $c_1  \leq 0$, we have
$4\omn^2  + c_1 \leq 4$, thus
\bee\begin{split}
\frac{\pp}{\pp t} \bar w  - \D_i \big ( \frac 1{H^2} \D_i  \bar w \big )  
&\leq   -  \frac 2{H^2 w}  |\nabla \bar w|^2+2 \gah^{-1} \bar w^2 +  4 \gah^{-1}   \bar w_+  + \gah^{-1} c_0(z,t)
\end{split}
\eee
with $c_0(z,t) := (2\omn^2  + c_1)_+$, $c_1=2 \gah/(n-1) - 2$, hence given by \eqref{eqn-co}. 
Let $\bw:=(w-1)_+=\bar w_+$. 
If we multiply the last inequality by $\bw^p=\bar w_+^p$, for some number $p \geq  0$,  and integrate by parts  (recalling that by \eqref{eqn-zero}
$\bw(\cdot,t)$ has compact support in $M_t$), we obtain
\begin{equation}\label{eqn-eqw222}
\begin{split}
\frac 1{p+1} \frac{d}{dt} \int_{M_t} \bw^{p+1} \, d\mu \leq  &- p  \, \int_{M_t} \frac 1{H^2} \bw^{p-1} |\D \bw|^2 \, d\mu 
 - 2 \int_{M_t} \frac 1{H^2 w} \bw^p  |\D \bw |^2 \, d\mu \\
&+ 2 \gah^{-1} \int_{M_t} \bw^{p+2} \, d\mu + (4 \gah^{-1}  +1) \int_{M_t} \bw^{p+1} \, d\mu + \gah^{-1}  
 \int_{M_t} c_0(z,t)  \, \bw^p \, d\mu. 
\end{split}
\end{equation}
Here we also used that $\ds {\partial} (d\mu) /{\partial t} = d\mu$.
Also, for  $p=0$ we use the inequality
$$ \int_{M_t}  \D_i \big ( \frac 1{H^2} \D_i  \bar w \big ) \chi_{\{ \bar w >0 \}} \, d\mu
= \int_{M_t \cap \partial \{ \bar w >0 \}}  \frac 1{H^2} \, \frac{\partial  \bar w}{\partial \nu}  \, d\sigma \leq 0.$$
We next remark that from the  definition of $w:= (\Fhn \, H)^{-1} \gah$, we may express 
${\displaystyle H^{-1} = \gah^{-1}  w \, \Fhn.}$
Also, $ {\displaystyle w\, \chi_{\{w > 1 \}}  \geq  (w-1)_+  = \bw.}$
Hence, we may combine  the two gradient terms on the right hand side of \eqref{eqn-eqw222} to conclude 
\bee
\begin{split}
\frac 1{p+1} \frac{d}{dt} \int_{M_t} \bw^{p+1} \, d\mu \leq &- (p+2) \,  \gah^{-2}   \int_{M_t} \Fhn^2 \bw^{p+1}  |\D \bw|^2 \, d\mu \\
&+ 2 \gah^{-1} \int_{M_t} \bw^{p+2} \, d\mu + (4\gah^{-1} +1)    \int_{M_t} \bw^{p+1} \, d\mu + \gah^{-1}
 \int_{M_t} c_0(z,t) \,  \bw^p \, d\mu.
 \end{split}
 \eee
Writing $$\bw^{p+1}  |\D \bw|^2 = \frac{4}{(p+3)^2} \,  |\D w^{\frac{p+3}2}|^2$$
and using \eqref{eqn-FFF} we obtain  
\bee
\begin{split}
\frac{d}{dt} \int_{M_t} \bw^{p+1} \, d\mu 
&\leq - 2 \, \lambda^2  \,  \gah^{-2} \gamma  \int_{M_t} |F|^2  |\D \bw^{\frac{p+3}2}|^2 \, d\mu  + c_2(p)\,   \gah^{-1} \int_{M_t} \bw^{p+2} \, d\mu 
\\ &+ c_1(p)\,   \gah^{-1}    \int_{M_t} \bw^{p+1} \, d\mu + \gah^{-1}
(p+1)  \int_{M_t} c_0 (z,t) \,  \bw^p \, d\mu
\end{split} 
\eee
for some new  positive constants $c_i(p)$ depending (linearly) on $p$ and  the initial data. This readily 
gives  \eqref{eqn-eqw3}  by setting $q:=(p+3)/2$. 

\end{proof}

We will next  prove the following variant   of Hardy's inequality adopted to our situation (see  in  \cite{Ca} and \cite{KO}
for  standard Hardy inequalities on complete non-compact manifolds).

\begin{prop}[Hardy Inequality]  Let $M_t$ be a solution of \eqref{eqn-IMCF} as 
in Theorem \ref{thm-ste}. Then, there exists a constant $C_n >0$ depending only on dimension $n$ 
such that  any function $g$ that is compactly supported on 
$M_t$, we have
\be\label{eqn-hardy}
\int_{M_t} g^2 \, d\mu \leq C(n)  \left (  \int_{M_t}  |\nabla g|^2 |F|^2 \, d\mu  + \int_{M_t}   g^2 |H| \, |F| \, d\mu \right ).
\ee
\end{prop} 
\begin{proof} 
To simplify the notation, set  $\rho(F):= |F| $  and recall that from our assumptions on $M_t$ we have  $\rho >0$ everywhere. 
We begin by computing $\Delta \rho$. We have
$$\Delta \rho = \D_i \D_i ( \langle F,F \rangle^{1/2} ) = \D_i ( \langle F,F \rangle^{-1/2} \, \langle \eei,F \rangle )=
\frac n{|F|} - \frac{|F^T|^2}{|F|^3} + H\, \frac {\Fn}{|F|}
$$
from which we conclude the lower bound 
\be\label{eqn-LaF}
\Delta \rho  \geq \frac{n-1}{\rho} - H.
\ee
Let   $g:=  \rho^{\gamma} \psi $ for some $\gamma <0$ to be chosen momentarily. 
We then have
$$|\D g|^2 = |\D (\rho^\gamma \psi) |^2 = | \gamma \rho^{\gamma-1} \psi \D \rho + \rho^\gamma \D \psi|^2 
\geq  2\gamma \rho^{2\gamma-1} \psi \D \rho \cdot \D \psi.$$
We next observe that it is convenient to choose $\gamma=-1/2$ which gives
$$|\D g|^2 \geq - \psi \rho^{-2} \D \rho \cdot \D \psi$$
or equivalently (using that  $\psi^2 =   g^2 \rho$) 
$$|\D g|^2 \rho^2  \geq - \frac 12 \,  \D \rho \cdot \D \psi^2= - \frac 12 \,  \D \rho \cdot \D (g^2 \rho).$$
After integrating by parts we obtain
$$\int_{M_t}  |\D g|^2 \rho^2 \, d\mu \geq 
\frac 12  \int_{M_t} g^2 \rho \,  \Delta \rho \,  d\mu.$$
Combining this with inequality \eqref{eqn-LaF} yields
\be\label{eqn-hardy2}
\int_{M_t}  |\D g|^2 \rho^2 \, d\mu \geq 
\frac {n-1}2  \int_{M_t} g^2   d\mu - \frac 12  \int_{M_t}  g^2 \rho \, H  d\mu
\ee
from which \eqref{eqn-hardy} readily follows. 
\end{proof}

We will now combine \eqref{eqn-eqw3} with the above  Hardy inequality to prove the following $L^{p+1}$ bound on $\bw$
in terms of its initial data.

\begin{thm}[$L^{p+1}$ estimate on $\bw$]\label{thm-lpestw}
Assume that $M_t$ is a solution to  \eqref{eqn-IMCF} as in Theorem \ref{thm-ste} 
 defined for 
$t \in (0,\tau]$, and assume that  $ \tau < T-3\delta$ with $T$ given by \eqref{eqn-T} and  $\delta >0$. Then, for any $p \geq 0$ there exists 
a constant $C=C(p)$ depending on $p, T, \delta$ and also on the constants $\kappa, \al_0$  such that 
\be\label{eqn-lte}
\sup_{t \in [0,\tau]}  \int_{M_t} \bw^{p+1}(\cdot,t) \, d\mu \leq C( p,\delta,T) \left (1+ \int_{M_0} \bw^{p+1}   \, d\mu \right ).
\ee
\end{thm}
\begin{proof}
We  recall  that $\gah(t)$ is a   solution of the ODE \eqref{eqn-ode4} with initial value $0 < \hat \gamma(0) < \gamma(0)$ 
so that  that its vanishing time $\hat T$ satisfies $T-2\delta < \hat T < T-\delta$, for the given small number $\delta >0$. 
For any number $p >0$, set $q:=(p+3)/2$. Applying   \eqref{eqn-hardy2} for $g=\bw^q$  gives 
\be\label{eqn-eqw4}
 \int_{M_t}  |\nabla \bw^q|^2 |F|^2 \, d\mu \geq \frac {n-1}2  \int_{M_t} \bw^{2q}    d\mu - \frac 12  \int_{M_t}  \bw^{2q} \,|F| H  d\mu.
 \ee
We will next estimate  $|F| H \chi_{\{\bw >0\}}$ in terms of $\bw$. 
Recall that by definition $w(\cdot,t):=\gah(t)  \,(\Fhn \, H)^{-1}$ and that from \eqref{eqn-FFF} we have 
$|F| \leq \la^{-1} \, \gamma^{-1/2}   \Fhn$.
Thus, 
$$|F| \, H \leq \la^{-1}  \gamma^{-1/2} \,   \Fhn H = \la^{-1}\, \gamma^{-1/2} \, \gah \, w^{-1}.$$
Since $w \,  \chi_{\{w >1\}} \leq (w-1) \, \chi_{\{w >1\}} = (w-1)_+ = \bw$, we have 
$$|F| \, H \,  \chi_{\{w >1\}} \leq \la^{-1}\, \gamma^{-1/2} \, \gah \, \bw^{-1}.$$
Thus \eqref{eqn-eqw4} yields
\be\label{eqn-hardy-appl}
 \int_{M_t}  |\nabla \bw^q|^2 |F|^2 \, d\mu \geq \frac {n-1}2  \int_{M_t} \bw^{2q}    d\mu - \frac 12 \la^{-1}\, \gamma^{-1/2} \, \gah
  \int_{M_t}  \bw^{2q-1}   d\mu.
  \end{equation}
Recall that $q=(p+3)/2$, so that $2q-1=p+2$. 
If we now combine this   last estimate with \eqref{eqn-eqw3} and also use that $\gah^{-1}\gamma >1$ and $n-1\geq 1$,
 we obtain   the differential inequality
\bee
\begin{split}
\frac{d}{dt} \int_{M_t} \bw^{p+1} \, d\mu  \leq &-  \lambda^2 \, \gah^{-1}   \int_{M_t} \bw^{p+3}   d\mu\,  +\\
&+ C(p) \, \gah^{-1} \left (   \int_{M_t} \bw^{p+2} \, d\mu +  \int_{M_t}  \bw^{p+1}\, d\mu  +  \int_{M_t} c_0(z,t) \,  \bw^p  d\mu \right )
\end{split}
\eee
for  constant $C(p)$
depending  on $p$ and also on $\kappa,\al_0$ and with $c_0(z,t)$ given by \eqref{eqn-co}.
We may  apply the interpolation  inequality
$$\int_{M_t} |g|^{p+2} \, d\mu = \int_{M_t} |g|^{\frac{p+1}2}\, |g|^{\frac{p+3}2} \, d\mu \leq  
\frac{\la^2}{2C(p)} \int_{M_t} |g|^{p+3} \, d\mu  + \frac {C(p)}{2\la^2}  \, \int_{M_t} |g|^{p+1} \, d\mu $$
to $g:=\bw$ to conclude that 
\be\label{eqn-eqw5}
\frac{d}{dt} \int_{M_t} \bw^{p+1} \, d\mu \leq  
C(p) \, \gah^{-1} \left (   \int_{M_t} \bw^{p+1} \, d\mu +  \int_{M_t} c_0(z,t) \,  \bw^p  d\mu \right ) 
\ee
for a new  constant $C(p)$ that depends on $p$ and also on our initial data $\alpha_0,\kappa$ and dimension $n$.

Because $M_t$ is non-compact, in order to  estimate the last term in \eqref{eqn-eqw5} in terms of ${\displaystyle \int_{M_t} \bw^{p+1} \, d\mu}$ 
we will need to look more carefully  into the coefficient $c_0(z,t)$.  We claim the following.

\begin{claim}\label{claim-XXX} Assume that  $\gah(t)$ is chosen so that  its vanishing time $\hat T$ satisfies $T-2\delta < \hat T \leq T-\delta$
for the given  small number $\delta >0$. Then, there exists a number $R_\delta \geq 1$ (depending on $\delta$) such that 
\be\label{claim1} c_0(\cdot,t) \equiv   0 \qquad \mbox{on} \,\,M_t \cap \{ |F| \geq R_\delta \}, \,\,\,   0 \leq t < \hat T. \ee 
\end{claim}
\begin{proof}[Proof of claim \ref{claim-XXX}] Recall that $c_0(\cdot,t):= 2 \, \big (\omn^2  +\gah(t)/(n-1) - 1 \big )_+$ and that 
$\gah(t) < \gamma(t)$ for all $t < \hat T$. Since by definition $\gamma(t)=(n-1)\, \al(t)^2/(1+ \al(t)^2)$ 
we may also express 
$\gah(t)=(n-1)\, \hat \al(t)^2/(1+\hat \al(t)^2)$ for some function of time function $\hat \al(t)$.  
It follows from the condition $T-2\delta < \hat T \leq T-\delta$ that
$$  0 <  \mu_1(\delta) \leq  \al(t) - \hat \al (t)  \leq \mu_2(\delta), \qquad \forall t < \hat T$$
for some positive constants $\mu_1(\delta), \mu_2(\delta)$ that tend to zero as $\delta \to 0$. 
Consider the cones defined by the graphs  $x_{n+1}= \al(t)\, |x|$ and  $x_{n+1} = \hat \al(t)\, |x| + \kappa$ over $x \in \R^n$. These  cones 
intersect at  $|x|= r(t):=\kappa/(\alpha(t)-\hat \alpha(t))$. Let $R(t):=\sqrt{1+\al^2(t)}\, r(t)$. 
It follows from \eqref{eqn-conet} and a simple geometric consideration 
that uses  the convexity of $M_t$ that 
$$ c_0(\cdot,t) \equiv 0 \qquad \mbox{on} \,\,M_t \cap \{ |F| \geq R(t) \}.$$ 
Since, 
$$R(t) :=  \kappa\, \frac{\sqrt{1+\al^2(t)}}{\alpha(t)-\hat \alpha(t)} \leq \kappa\, \frac{\sqrt{1+\al_0^2}}{\mu_1(\delta)}:=R_\delta,  \,\,\,   0 \leq t < \hat T$$
the claim  follows. 
\end{proof}

Using the above  claim and the bound $c_0(z,t) \leq 2$,  we  may  now  estimate  the term ${\displaystyle \int_{M_t} c_0(z,t) \,  \bw^p  d\mu}$ in \eqref{eqn-eqw5} as 
\bee
\begin{split}
\int_{M_t} c_0(z,t) \,  \bw^p  d\mu  &\leq  2 \int_{M_t \cap \{c_0 >0 \}} \bw^p\, d\mu
\leq C(R_\delta,p) \left (  \int_{M_t} \bw^{p+1}\, d\mu \right )^{p/(p+1)} \\
&\leq C(R_\delta,p) \left (  \int_{M_t} \bw^{p+1}\, d\mu + 1\right ).
\end{split}
\eee
Combining the last estimate with \eqref{eqn-eqw5}, we obtain
\be\label{eqn-eqw7}
\frac{d}{dt} \int_{M_t} \bw^{p+1} \, d\mu \leq  C(p, \delta) \, \gah(t)^{-1} \left (  1+ \int_{M_t} \bw^{p+1} \, d\mu \right ).
\ee
Since we have assumed that $\gah(t)$ vanishes at $\hat T$ and $T-2\delta < \hat T < T-\delta$,   it follows that $\gah(t)^{-1} \leq C(\delta)$  
for all $t < T-3\delta$. We conclude from \eqref{eqn-eqw7} that
\bee
\frac{d}{dt} \int_{M_t} \bw^{p+1} \, d\mu \leq  
C(p, \delta) \,  \left (  1+ \int_{M_t} \bw^{p+1} \, d\mu \right )
\eee
for another constant $C(p,\delta)$. After integrating this inequality in time $t$ we conclude that if $\tau < T-3\delta$, \eqref{eqn-lte} holds. 
\end{proof}

%

\section{$L^\infty$ estimates  on  $1/H$}\label{sec-Linfty-below}

We will assume  throughout this  section  that $M_t$ is an entire graph convex  solution of   \eqref{eqn-IMCF} on $[0,\tau]$
as in   Theorem \ref{thm-ste} and that $\tau < T-3\delta$,  for some $\delta >0$, where $T$ is the number given by \eqref{eqn-T}. 
We  will establish  a local $L^\infty$  bound  on $(\Fhn \, H)^{-1}$  which holds on $M_t$  for all $t \in [0,\tau]$ and depends 
only on the initial data, on  $T$ and on $\delta$. 
This bound  constitutes the main  step in the proof of the long time  existence result Theorem \ref{thm-main}. 
It states as follows.

\begin{thm}[$L^\infty$  bound on $w$ in terms of its spatial averages] \label{thm-Hbelow0}
Assume that $M_t$ is a solution to  \eqref{eqn-IMCF} as in Theorem \ref{thm-ste} 
 defined for 
$t \in (0,\tau]$, and assume that  $ \tau < T-3\delta$ with $T$ given by \eqref{eqn-T} and  $\delta >0$.
There exist absolute constants $\mu >0$ and $\sigma >0$  
 and a constant $C$ that depends  on  $\alpha_0, \kappa$,  on $\delta$,
and the initial  bound $\sup_{M_0} \Fom \, H$,    
for which  $w:= \gah(t)  \,(\Fhn \, H)^{-1}$ satisfies the bound 
\be\label{eqn-linfty}
\sup_{t \in (t_0, \tau]} \|w\|_{L^\infty(M_t)} \leq C \, {t_0}^{-\mu} \left (1 + \sup_{t \in ({t_0}/4, \tau]} \sup_{R \geq 1} 
R^{-n} \int_{M_t \cap \{ |F| \leq R \}} w(\cdot,t)\, d\mu \right )^{\sigma}.
\ee
for any $t_0 \in (0,\tau/2]$. 
\end{thm}

For the proof of this theorem we will use  a parabolic variant of Moser's  iteration on the differential inequality \eqref{eqn-eqw2} that is satisfied by
$w:= \gah(t)  \,(\Fhn \, H)^{-1}$. Such technique was first introduced in the nonlinear parabolic context by Dahlberg  and Kenig
in \cite{DK}.  In fact we will closely follow the proof in \cite{DK} (see also in the proof of Lemma  1.2.6 in \cite{DKb}). For the inverse mean curvature flow in the  compact setting,  a similar bound was shown in \cite{HI} via a variant of the 
Stampacchia iteration method.

The estimate \eqref{eqn-linfty} will be shown in two steps  Propositions  \ref{prop-Hbelow} and \ref{prop-Hbelow2} below.
Let us begin by introducing some notation. For any given number $t_0 \in (0,\tau]$ we set
$$S_{t_0}:= \{ (P ,t) \in \R^{n+1} \times (0,t_0]: \,\,  P  \in   M_t, \,\, t \in  (0, t_0] \, \} = \cup_{t \in (0,t_0]} M_t \times \{t\}.$$
Also, for any  given numbers  $\rho_0 > 1$, $t_0 \in (0,\tau]$ and $r \in (0,1)$
we consider the cylinders  in  $\R^{n+1} \times (0,+\infty)$ given by
$$Q_{\rho_0,t_0}^r:= \{ (\bx,t) \in \R^{n+1} \times (0,+\infty):\,\,  \rho_0 (1-r) <   |\bx| < \rho_0 (1+r), \,\, (1- r)  \, t_0 < t \leq t_0 \}.$$ 
In particular, we set 
$$Q_{\rho_0,t_0 }:=Q_{\rho_0,t_0}^{1/4}, \qquad Q_{\rho_0,t_0 }^{*}:=Q_{\rho_0,t_0}^{1/2}, \qquad Q_{\rho_0,t_0 }^{**}
:=Q_{\rho_0,t_0}^{3/4}.$$
Notice that since in equation \eqref{eqn-u0} one cannot scale the time $t$, it is not necessary to use the  standard parabolic scaling 
in the above cylinders,  one can just  use the same scale  in  $\bx$ and $t$.

\begin{prop}\label{prop-Hbelow} Assume that $M_t$ is a solution to  \eqref{eqn-IMCF} as in Theorem \ref{thm-ste} 
 defined for 
$t \in (0,\tau]$, and assume that  $ \tau < T-3\delta$ with $T$ given by \eqref{eqn-T} and  $\delta >0$.
There exist absolute constants $\mu >0$ and $\sigma >0$ and a constant $C$ that depends  on  $\alpha_0, \kappa$,  on $\delta$,
and the initial  bound $\sup_{M_0} \Fom \, H$,    
for which  $w:= \gah(t)  \,(\Fhn \, H)^{-1}$ satisfies the bound
\be\label{eqn-Hbelow}
\|w\|_{L^\infty( Q_{\rho_0,t_0} \cap S_{t_0})} \leq C \, {t_0}^{-\mu} 
\left (1 + \sup_{t \in (t_0/4, {t_0}]} \rho_0^{-n} \int_{M_t \cap  Q^{**}_{\rho_0,t_0}} w(\cdot,t)\, d\mu \right )^{\sigma}
\ee
which holds for any $\rho_0 >2$ such that $Q_{\rho_0,t_0} \cap S_{t_0}$ is not empty. 
\end{prop}

\begin{rem} For the remaining of this section  we will call {\em  uniform constants} the   constants that may depend on the number  $\delta >0$,
the constants $\alpha_0, \kappa$,  but that are {\em independent} of $\rho_0$ and ${t_0}$. 
\end{rem}

Since  $w$ satisfies the differential inequality  \eqref{eqn-eqw2},  if we set $\brw:=\max (w,1)$
it follows that $\brw$ satisfies the same differential inequality and
since $\brw \leq   \brw^2$  we have 
\be\label{eqn-nice3}
\frac{\pp}{\pp t} \brw  - \D_i \big ( \frac 1{H^2} \D_i  \brw \big ) \leq  -  \frac 2{H^2 \brw}  |\nabla \brw|^2  + c_2\,  \gah^{-1}  \brw^2. 
\ee
for some new constant $c_2>0$. 

\begin{rem} In the following we shall not distinguish between the image $F(z, t)$  of a point $ z \in M$  and its  coordinate vector  in $\R^{n+1}$. 
\end{rem}

For the  given numbers  $\rho_0 >1$  and $t_0 \in (0,\tau]$ and any numbers $1/4 <   r < \br < 1/2$, we consider a
radial  cutoff function $\psi= \psi (\rho,t), \rho=|\bx|$, $\bx \in \R^{n+1}$  with   $\psi  \in C^\infty_c (Q_{\rho_0,t_0}^{\br})$ satisfying 
\be\label{eqn-psi1}
\psi \equiv 1 \quad \mbox{on}\,\, Q_{\rho_0,t_0}^{r}, \quad 0 \leq \psi \leq 1, \quad \rho_0 \, | \psi_\rho | + t_0 \, | \psi_t| \leq C \, (\br - r)^{-1}. 
\ee
We extend $\psi$ to be equal to zero outside $ Q_{\rho_0,t_0}^{\br}$ and define the function $\eta$ on $S_{t_0}$ by 
\be\label{eqn-eta1}
\eta  (F,t):= \psi (|F|,t), \qquad F \in M_t.
\ee

\begin{lem}\label{lem-energy20}   Under the assumptions of Theorem  \ref{prop-Hbelow}, 
for any $p \geq 1$ and $\theta:=(p+2)/2$, we have 
 \be\label{eqn-energy20} 
\bs
\sup_{t \in (0,{t_0}]}  \int_{M_t}  (\eta^2   \brw^p) (\cdot,t)  \, d\mu  +  & \int_0^{t_0} \int_{M_t}   \rho_0^2 \,  |\nabla (\eta  \brw^\theta)  |^2   d\mu \, dt  \\ &\leq  
C \, t^{-1}_0 \, (\br-r)^{-2} \,  \int_0^{t_0} \int_{M_t \cap \{ \eta >0 \}}  \brw^{2\theta}  d\mu\,   dt 
\end{split} \ee
where $\eta \in C_c (S_{t_0})$ is the cutoff function defined by \eqref{eqn-eta1}
\end{lem}

\begin{proof}
We begin by observing that  the cutoff function defined by \eqref{eqn-eta1} satisfies 
\be\label{eqn-beta}
|\nabla \eta | \leq C \, \rho^{-1}_0 \, (\br-r)^{-1} \qquad \mbox{and} \qquad |\partial_t \eta| \leq  C \,  (\br-r)^{-1} \left  ( \,  \rho^{-2}_0 \,  H^{-1}  |\Fn | + t_0^{-1}  \right   ) 
\ee
where we have denoted by $\nabla \eta$ the  gradient of $\eta $ on $M$. 
The first inequality follows from \eqref{eqn-psi1} and  the calculation
$$ |\nabla_i \eta  |  = \frac 12  \, |\psi_\rho| \,|F|^{-1} \,  | \nabla_i \langle F, F \rangle | 
\leq C \, \rho^{-1}_0 \,  (\br - r)^{-1} $$
while the second inequality follows from  \eqref{eqn-psi1} and the calculation
$$ | \partial_t \eta  | =  |\psi_\rho | \, |F|^{-1}\,  | \langle F, F_t \rangle | + |\psi_t| 
\leq C\,  (\br-r)^{-1}  \left ( \rho^{-2}_0  \, H^{-1}  |\Fn | + t_0^{-1} \right ).$$
Using  equation \eqref{eqn-nice3} and  that $\partial ( d \mu )/ \partial t = d\mu$, $\brw \geq 1$  and $H^{-1}= \gah^{-1} \Fhn \, w$, we have
\bee
\bs
 \frac d{dt}  \int  \brw^p \eta^2 \, d\mu &= p  \int \brw^{p-1}\,  \eta^2  \, \brw_t  \,  d\mu  + 2 \int \brw^p \eta \, \eta_t   \, d\mu
+    \int \brw^p \, \eta^2 \, d\mu\\
&\leq - p (p-1)   \int  \frac 1{H^2}\,  \brw^{p-2}  \eta^2 \,   |\nabla \brw|^2    \, d\mu -  2 p \,  \int  \frac 1{H^2} \brw^{p-2} \eta^2 |\nabla \brw|^2   \, d\mu\\ 
& -   2 p  \int \frac 1{H^2}  \, \brw^{p-1}  \eta  \nabla_i \brw \nabla_i \eta  \, d\mu  + 2 \int \brw^p  \eta \, |\eta_t |   \, d\mu +  c_2 \, p  \, \gah^{-1} 
 \int \brw^{p+1} \, \eta^2 \, d\mu +  \int \brw^p \, \eta^2 \, d\mu\\
&\leq - p (p+1) \gah^{-2}    \int \Fhn^2 \, \brw^p \eta^2 \,   |\nabla \brw|^2    \, d\mu -   2 p\,  \gah^{-2}  \int \Fhn^2 \, \brw^{p+1}  \eta  \nabla_i \brw \nabla_i \eta  \, d\mu\\
&+ 2  \int \brw^p \eta \, |\eta_t|   \, d\mu +  \bar c_2 \,    \int \brw^{p+1}\, \eta^2 \, d\mu. 
 \end{split}
\eee
with $\bar c_2 := c_2 \, p \,  \gah^{-1} +1$. 
Let $\theta =(p+2)/2$. Writing
$$\brw^p \,   |\nabla \brw|^2 = \theta^{-2} \,|\nabla \brw^\theta|^2 \qquad \mbox{and} \qquad  \brw^{p+1} \nabla \brw = \theta^{-1} \brw^\theta
\, \nabla \brw^\theta$$
we obtain 
\bee
\bs
 \frac d{dt}  \int \brw^p \eta^2 \, d\mu &\leq - p (p+1)  \theta^{-2} \gah^{-2} \int  \Fhn^2   \eta^2   |\nabla \brw^\theta |^2    \, d\mu 
 \\ &- 2 p\,  \theta^{-1}  
 \gah^{-2} \int \Fhn^2  \, \eta\,  \brw^\theta |\nabla \eta| \, |\nabla \brw^\theta| \, d\mu 
 + 2  \int \brw^p \eta \, |\eta_t|   \, d\mu +   \bar c_2 \,   \int \brw^{p+1}\, \eta^2 \, d\mu. 
 \end{split}
 \eee
We estimate
$$\int   \Fhn^2 \brw^\theta \eta |\nabla \eta | \, | \nabla \brw^\theta |  \, d\mu \leq \frac {(p+1)}{4\theta} \int  
\Fhn^2 \eta^2   |\nabla  \brw^\theta |^2    \, d\mu
+ \frac \theta{(p+1)}  \, \int \brw^{2\theta} \Fhn^2 |\nabla \eta|^2 \, d\mu$$
to conclude 
\be\label{eqn-v2} 
\bs
 \frac d{dt}  \int \brw^p \eta^2 \, d\mu &+ \frac  12 \,  p\,(p+1) \theta^{-2}  \gah^{-2}    \int  \Fhn^2 \eta^2 \,   |\nabla  \brw^\theta |^2    \, d\mu \\
& \leq  C  \, \left (  \gah^{-2} \frac p{p+1}  \int  \Fhn^2 \brw^{2\theta} |\nabla \eta|^2 \, d\mu+  \gah^{-1} \int \brw^{p+1}  \eta^2   \, d\mu 
+ \int \brw^p \eta \, |\eta_t|    \, d\mu \right ) 
 \end{split}
 \ee
for a uniform constant  $C$ that is in particular independent of $p$.  
Also,  by \eqref{eqn-beta} we have 
$$\int \brw^p \eta \, |\eta_t|    \, d\mu \leq  C\,  (\br - r)^{-1}  \left ( t_0^{-1} \, \int \brw^p  \eta   \, d\mu +  \rho^{-2}_0  \, 
\int \brw^p \eta \,  \, H^{-1}  |\Fn | \, d\mu \right ).$$
Using that 
$$H^{-1} = \gah^{-1} \, \Fhn \, w \leq \gah^{-1} \, \Fhn \, \brw$$
and
$$\Fhn \, |\Fn| \leq C\, |F|^2 \leq C\, \rho_0^2$$
on the support of $\eta$  and also that $\brw \geq 1$,  we obtain the bound 
\be\label{eqn-eta12}
\bs
\int \brw^p \eta \, |\eta_t|    \, d\mu &\leq  C\, (\br-r)^{-1} \left ( t_0^{-1} \int \brw^p \, \eta    \, d\mu +  
\int \gah^{-1}  \, \brw^{p+1}  \, \eta  \, d\mu \right )\\ &\leq C\, t^{-1}_0  (\br-r)^{-1}  \int  \gah^{-1}  \, \brw^{p+1}  \, \eta  \, d\mu.
\end{split}
\ee
Since $p \geq 1$, we have 
$$\frac 49  \leq  p  (p+1)  \theta^{-2} = \frac{4 p  (p+1) }{(p+2)^2}  \leq 4.$$
Integrating \eqref{eqn-v2} in time 
on $(0,t]$ for all $t \in (0,{t_0}]$ and using \eqref{eqn-beta} and  \eqref{eqn-eta12}
yields
\bee
\bs
\sup_{t \in (0,{t_0}]}  & \int_{M_t } \brw^p \, \eta^2\, d\mu  \, dt  +  \gah^{-2}  \int_0^{t_0} \int_{M_t}
  \Fhn^2 \eta^2   |\nabla  \brw^\theta |^2    \, d\mu \, dt \\
& \leq  C    (\br-r)^{-2}  \left (  \rho_0^{-2}   
\int_0^{t_0} \int_{M_t \cap \{ \eta >0 \}}    \gah^{-2}  \Fhn^2 \brw^{2\theta} d\mu \, dt  +   t_0^{-1} 
\int_0^{t_0} \int_{M_t \cap \{ \eta >0 \}}     \gah^{-1}  \brw^{p+1}   d\mu \, dt \right ).  
\end{split} 
\eee 
Using  the bounds  $\bar w \geq 1$, \eqref{eqn-FFF} and  $   c\, \rho_0^2 \leq \Fhn^2 \leq  |F|^2 \leq C \rho_0^2$ 
and $\gah^{-1} \leq C_\delta$, for ${t_0} < T-3\delta$, we obtain
\bee
\bs
\sup_{t \in (0,{t_0}]}  \int_{M_t}  (\eta^2   \brw^p) (\cdot,t)  \, d\mu  + 
&\int_0^{t_0} \int_{M_t}  \rho_0^2 \eta^2 \,   |\nabla  \brw^\theta |^2   d\mu \, dt  \\
&\leq  
C \, t_0^{-1} \, (\br-r)^{-2} \int_0^{t_0} \int_{M_t \cap \{ \eta >0 \}}   \brw^{2\theta}  d\mu\,   dt. 
\end{split} \eee
Finally, using the estimate 
\bee
\bs
\int_0^{t_0} \int_{M_t}   \rho_0^2 \, \eta^2  |\nabla  \brw^\theta |^2   d\mu \, dt & \geq 
\frac 12 \int_0^{t_0} \int_{M_t}   \rho_0^2 \, |\nabla (\eta  \brw^\theta)   |^2   d\mu \, dt  - 
 4 \int_0^{t_0} \int_{M_t}  \rho_0^2 \, |\nabla \eta |^2   \brw^{2\theta}    d\mu \, dt \\
 & \geq \frac 12  \int_0^{t_0} \int_{M_t}    \rho_0^2 \, |\nabla (\eta  \brw^\theta)  |^2   d\mu \, dt  - C\, (\br - r)^{-2} 
\int_0^{t_0} \int_{M_t \cap \{ \eta >0 \}}     \brw^{2\theta}   d\mu \, dt
 \end{split}
 \eee
we conclude \eqref{eqn-energy20}. 

\end{proof}

We will prove next  a  variant of  the  following  Sobolev inequality which holds on any complete manifold $N^n$, with  $n \geq 3$
\be\label{eqn-sob0}
\left ( \int_{N^n} |f|^{\frac {2n}{n-2}}  \, d\mu \right )^{\frac {n-2}n}
\leq C(n)  \, \int_{N^n} |\nabla f|^2 + H^2  f^2 \, d\mu
\ee
and for any  $f \in C^1_c(N^n)$. When $n=2$ we will use instead the inequality
\be\label{eqn-sob00}
\left ( \int_{N^2} |f|^4 \, d\mu \right )^{1/2}
\leq C \, |N^2 \cap  \mbox{supp} f  |^{1/2}  \, \int_{N^2} |\nabla f|^2 + H^2  f^2 \, d\mu
\ee
which holds for any $f \in C^1_c(N^2)$.

\begin{lem}\label{lem-sob}
We  set $q^*:= q/(q-1)$ with $q=n/2$ if $n \geq 3$ and $q=2$ if $n=2$. Then, for any $k \in (0,q^*)$
and $h \in C_c^{1,0}(S_{t_0})$  we have
\be\label{eqn-sob1} 
\begin{split}
\int_0^{t_0}  \int_{M_t} h^{2k} \, d\mu\, dt   \leq \,  &C \left \{   \int_0^{t_0} |M_t \cap \mbox{\em supp} \, h|^\lambda  \int_{M_t}  |\nabla h|^2  +  
H^2 h^2 \, d\mu\, dt
\right .\\
& \left . \cdot \sup_{t \in (0,{t_0}]} \left ( \int_{M_t}   h^{2(k-1)q}(\cdot,t) \, d\mu \right )^{1/q} \right \}.
\end{split}
\ee
with $\lambda=0$ if $n \geq 3$ and $\lambda=1/2$ if  $n=2$. 
\end{lem}

\begin{proof}
Since $h(\cdot,t) \in C_c^{1}(M_t)$, it  follows from \eqref{eqn-sob0} that for $n \geq 3$ and any $ t \in (0,t_0]$ we have 
$$\left ( \int_{M_t } |h|^{q^*}  \, d\mu \right )^{2/{q^*}} \leq C(n) \,
 \, \int_{M_t} |\nabla h|^2  + H^2  h^2 \, d\mu.$$
Hence, for any  $t \in (0,{t_0}]$ we have 
\bee
\bs
\int_{M_t } & h^{2k} \, d\mu  =  \int_{M_t } h^2\, h^{2(k-1) } \, d\mu \\
&\leq \left ( \int_{M_t } |h|^{2q^*}  \, d\mu \right )^{1/{q^*}} \, \left (    \int_{M_t} |h|^{2(k-1)q}  \, d\mu \right )^{1/q} \\
& \leq C \left (  \int_{M_t}    |\nabla h|^2  + 
H^2   h^2 \, d\mu
\right ) \,  \sup_{t \in (-0,{t_0}]} \left (  \int_{M_t }   h^{2(k-1)q}(\cdot,t) \, d\mu \right )^{1/q}.
\end{split}
\eee
Inequality \eqref{eqn-sob1} with $\lambda =0$  now follows by  integrating in $t$. When $n=2$ one uses the same
calculation as above with the only difference that  now $q^*=2$ and by \eqref{eqn-sob00} we have 
$$\left (\int_{M_t } |h|^{2q^*}  \, d\mu \right )^{1/{q^*}} \leq  C \, |M_t  \cap  \mbox{supp} \, h  |^{1/2}  \, \int_{M_t} |\nabla f|^2 + H^2  f^2 \, d\mu$$
leading to \eqref{eqn-sob1} with $\lambda =1/2$.

\end{proof}

We will  next combine \eqref{eqn-energy20} and \eqref{eqn-sob1} to conclude the proof of  Proposition \ref{prop-Hbelow} via a 
Moser iteration argument. 

\begin{proof}[Proof of Proposition  \ref{prop-Hbelow}]  
For the  given numbers  $\rho_0 >1$  and $t_0 \in (0,\tau]$ and any numbers $1/4 <   r < \br < 1/2$,
we let $\eta \in C_c (S_{t_0} \cap Q^{**}_{\rho_0,r_0})$ be the cutoff function given by \eqref{eqn-psi1}-\eqref{eqn-eta1}. 
Clearly,  
$$   \frac{\rho_0}2 \leq   |F|  \leq 2\, \rho_0,   \qquad \mbox{on}  \,\,\,   S_{t_0} \cap Q^*_{\rho_0,r_0}$$
which includes the support of $\eta$. 

We first  apply \eqref{eqn-sob1} to $h:=\eta \, \brw^\theta \in C^{1,0}(S_{t_0})$  to obtain
\be\label{eqn-sob4} 
\begin{split}
\rho_0^{-n} \int_0^{t_0}  \int_{M_t}  (\eta \brw^\theta)^{2 k} \, d\mu  \leq  &C \left \{  \rho_0^{-n} \int_0^{t_0}  \int_{M_t}
 \rho_0^2  \, |\nabla (\eta \brw^\theta)|^2  +  
\rho_0^2  \, H^2 (\eta   \brw^\theta)^2  \, d\mu\, dt
\right .\\
& \left . \cdot \sup_{t \in (0,{t_0}]} \left ( \rho_0^{-n} \, \int_{M_t}   (\eta \brw^\theta)^{2(k-1)q}(\cdot,t) \, d\mu \right )^{1/q} \right \}.
\end{split}
\ee
Notice that we have multiplied by $\rho^{-n}_0$ to make the inequality scaling invariant in space.  For $n \geq 3$ inequality \eqref{eqn-sob4}
simply follows from \eqref{eqn-sob1}, since  $\lambda=0$ and  $q=n/2$. When $n=2$ we apply  \eqref{eqn-sob1} with 
$\lambda=1/2$ and $q=2$ and use   the fact that $|M_t \cap \mbox{supp} \, \eta| \leq C\, \rho^2_0$.

From Proposition \ref{prop-Habove}  we have 
\be\label{eqn-boundHH}
\rho_0^2\,  H^2 \leq C\,  |F|^2 H^2 \leq C
\ee
since $|F| \geq\rho_0/2$  on the support of $\eta$. 
We next choose $k=k(p) >1$ such that 
$$2\theta (k-1)q = p.$$ Since $\theta = (p+2)/2$ this means that
$(p+2)  (k-1)q = p$, hence $(k-1)q = p/(p+2) <1$ or  $k < (q+1)/q < q/(q-1):=q^*$. 
In addition $k > 1+ p/(p+2)q \geq 1+1/(3q)$, since $p \geq 1$. Summarizing, for future reference we have
\be
1 + \frac 1{3q} < k=  k(p) < q^{*}
\ee
with $q, q^*$ as in Lemma \ref{lem-sob}. Thus, from \eqref{eqn-sob4} and \eqref{eqn-boundHH} we obtain
\be\label{eqn-sob2} 
\begin{split}
\rho_0^{-n} \int_0^{t_0}  \int_{M_t}  \eta^2  \brw^{2\theta k}  \, d\mu  \leq  &C \left \{  \rho_0^{-n} \int_0^{t_0}  \int_{M_t}
 \rho_0^2  \, |\nabla (\eta \brw^\theta)|^2  +  
 \eta^2    \brw^{2\theta}   \, d\mu\, dt
\right .\\
& \left . \cdot \sup_{t \in (0,{t_0}]} \left ( \rho_0^{-n} \, \int_{M_t}   \eta^{2(k-1)q} \brw^p(\cdot,t) \, d\mu \right )^{1/q} \right \}.
\end{split}
\ee
To simplify the notation, set
$$S^r_{\rho_0,t_0} := Q^r_{\rho_0,t_0} \cap S_{t_0} \qquad \mbox{and} \qquad  S_{\rho_0,t_0}^\br := Q^\br_{\rho_0,t_0} \cap S_{t_0}$$
and recall that from its definition $\eta \equiv 1$ on $Q^r_{\rho_0,t_0}$ and $\eta \equiv 0$ outside $ Q^{\br}_{\rho_0,t_0}$.
Also,  set
$$B:=  (\br-r)^{-2} \, t_0^{-1}$$
Combining  \eqref{eqn-energy20} and \eqref{eqn-sob2} yields
\be\label{eqn-com1}
\rho_0^{-n} \iint_{S^r_{\rho_0,t_0}}  \brw^{2\theta k} \, d\mu dt \leq C \, B^{1+1/q} \, t_0^{-(1+1/q)}  \left (  \rho_0^{-n} \iint_{S^\br_{\rho_0,t_0}}    
 \brw^{2\theta}   \, d\mu  dt \right )^{1+1/q}.
\ee 
 
\smallskip

We will now iterate this inequality to obtain the desired $L^\infty$ bound on $\brw$. 
To this end, we   define $p_0, p_1, \cdots$ and $\theta_0, \theta_1, \cdots$  by letting  $p_0=1$
and setting 
\be\label{eqn-iter0}
\theta_\nu = \frac{p_\nu+2}2, \qquad \theta_{\nu+1} = k_\nu \, \theta_\nu,  \qquad k_\nu=k_\nu(p_\nu)  = 1+ \frac{p_\nu}{(p_\nu+2)\, q}.
\ee
We also define
$$\rho_\nu := \frac{(1+\nu)}{2(1+2\nu)} \qquad  
\mbox{and} \qquad  Q_\nu:=Q^{r_\nu}_{\rho_0,t_0}, \quad S_\nu:=S^r_{\rho_0,t_0}=Q^{\rho_\nu}_{\rho_0,t_0}\cap S_{t_0}.$$
Observe that under this notation $Q_0=Q^*_{\rho_0,t_0}$ while $\lim_{\nu \to \infty} Q_\nu =Q^{*}_{\rho_0, t_0}$.  
Also, set 
$$M_\nu :=  \left ( \rho_0^{-n} \iint_{S_\nu}   v^{2\theta_\nu}   \, d\mu\, dt \right )^{1/2\theta_\nu}.$$
It then follows from \eqref{eqn-com1} that
\be\label{eqn-iter1}
M_{\nu+1}^{2\theta_{\nu+1}} \leq C\, B_\nu^{1+1/q} \,  M_\nu^{2\theta_\nu (1+1/q)}
\ee
with 
$$B_\nu=(r_\nu-r_{\nu+1})^{-2} \, t_0^{-1}  \leq C\, \nu^4 \, t_0^{-1}.$$
Since $q >1$, it follows from \eqref{eqn-iter1} that 
\be\label{eqn-iter3}
M_{\nu+1} \leq (C\, \nu^8 \, {t_0}^{-2})^{1/2\theta_{\nu+1}}  \,  \, M_\nu^{\lambda_\nu}
\ee
with
$ \lambda_\nu =  (1+1/q) / k(p_\nu)$.
Since $\lim_{\nu \to \infty} p_\nu = +\infty$ we have  $\lim_{\nu \to \infty} k(p_\nu) = 1+1/q$. It follows that  
$$E^\nu < \theta_\nu < (E^*)^\nu \qquad \mbox{and} \qquad  E^\nu < p_\nu < (E^*)^\nu$$
for some numbers $1 < E < E^* < \infty.$
Also, $1 <  \bar \lambda_\nu < 1+ C\, E^{-\nu}.$
We conclude from the bounds above that 
$$\lim_{\nu \to \infty} M_\nu \leq C \, t_0^{-\mu_0}  M_0^{\sigma_0}$$
for some absolute  constants $ \mu_0$ and $\sigma_0$. Thus,  
\be\label{eqn-linfty1}
\|\brw \|_{L^\infty(Q_{\rho_0,t_0} \cap S_{t_0})} \leq C \, t_0^{-\mu_1}  \left ( \rho_0^{-n}  \iint_{M_t \cap  Q^*_{\rho_0,t_0}}  
 \brw^3   \, d\mu\, dt \right )^{\sigma_1}
\ee
with    $2\theta_0=p_0+2=3$ and for  some new positive absolute constants $\mu_1$ and $\sigma_1$. The constant $C$ is independent of $\rho_0$ and $t_0$.  

To finish the proof of  the proposition  it will be sufficient to estimate the integral  on  the right hand side of \eqref{eqn-linfty1}   in terms of 
$$I:= \sup_{t \in (t_0/4,{t_0}]} \left ( \rho_0^{-n}  \int_{M_t \cap Q^{**}_{\rho_0,t_0}} \brw(\cdot,t)\, d\mu \right ).$$
To this end, we set again  $B:= (\bar r-r)^{-2} \, t_0^{-1}$
and  combine \eqref{eqn-sob4} with \eqref{eqn-energy20} and the bound \eqref{eqn-boundHH},
to obtain  for  $\theta_0:=3/2$ the bound 
\be
\begin{split}\label{eqn-com5} 
\left (\rho_0^{-n} \iint_{S^r_{\rho_0,t_0}} \brw^{2\theta_0 k} \, d\mu dt  \right ) &\leq C \, B  \left ( \rho_0^{-n}   \iint_{S^\br_{\rho_0,t_0}}    
 \brw^{2\theta_0}   \, d\mu\, dt  \right ) \\ & \cdot   \sup_{t \in ((1-\br)\, t_0,{t_0}])} 
 \left ( \rho_0^{-n} \int_{M_t \cap Q^\br_{r_0,t_0}}   \brw^{2\theta_0 (k-1)q}(x,t) \, d\mu \right )^{1/q}.
 \end{split}
\ee
If we  choose $k >1$ so that $2\theta_0 (k-1)q =1$, the above bound yields  
\be\label{eqn-iter10}
\rho_0^{-n} \iint_{S^r_{\rho_0,t_0}} \brw^{2 \theta_0 k } \, d\mu \, dt  \leq C \, B  \, I^{1/q} \,  \left ( \rho_0^{-n}   \iint_{S^\br_{\rho_0,t_0}}    
\brw^{2\theta_0}   \, d\mu\, dt \right ). 
\ee
Setting, 
$$m(r,k):= \iint_{S^r_{\rho_0,t_0}} \brw^{2 \theta_0 k } \, d\mu \, dt$$
 follows from \eqref{eqn-com5} that for any $1/4 < r < \bar r < 3/4$, we have 
\be\label{eqn-mrk}
m(r,k ) \leq C \, (\bar r - r)^{-2} \bar {t_0}^{-1} \, I^{1/q} \, m(\bar r,1).
\ee
Using H\"older's inequality we have
$$m(\bar r,1) \leq m(\bar r,k)^{\lambda/k} \, m(\bar r,s)^{(1-\lambda)/s}$$
for any $ s \in (0,1)$ with $\lambda=\frac{(1-s)k}{k-s}$. For $\gamma >1$ and $r \in [2/3,1]$, \eqref{eqn-mrk}  shows that
\bee
\begin{split}
\log m( 3 r^\gamma /4,k) \leq \log C + \log {t_0}^{-1} &+ \frac 1q \log I + \log ( 3 \, (r-r^\gamma)/ 4)^{-2}\\
&+ \frac{\lambda}k  \log m(3r/4,k) + \frac{1-\lambda}q \,  \log m(3/4,s)
\end{split}
\eee
since $m(3r/4,q) \leq m(3/4,q)$. Integrating in $r$ with respect to $dr/r$ on $[2/3,1]$  we find after a change of variable that
\be\label{eqn-mrk2}
\begin{split}
\gamma^{-1} \int_{2/3}^1 &\log m(3r/4,k) \, \frac{dr}r \\
& \leq C_1 \log I + C_2 \log {t_0}^{-1} + C_2 \log m(3/4,s) +
 C_3 + \frac  \lambda k \int_{2/3}^1 m( 3r/4, k) \, \frac{dr}r. 
 \end{split}
\ee
Now choose $s$ so that $2\theta_0 s  = 1$ (recall that we have set $\theta_0=3$)  and $\gamma$ so close to $1$ so that $\gamma^{-1} > \lambda /k$. 
If  $m(1/2,k) \leq 1$,
then since $k >1$ we conclude that $m(1/2,1) \leq C$ and the bound ${\ds \|\brw \|_{L^\infty(Q_{\rho_0,t_0} \cap S_{t_0})} \leq C}$ follows from  \eqref{eqn-linfty1}.
Otherwise, $\log m(3r4,k) >0$ for $r \in [2/3,1]$ and from \eqref{eqn-mrk2} we obtain 
$$
(\gamma^{-1} - \frac \lambda k )  \int_{2/3}^1 \log m(3r/4,k) \, \frac{dr}r  \leq C_1 \log I + C_2 \log  {t_0}^{-1} + C_2 \log m(3/4,s) + C_3 
$$
which yields
$$m(1/2,k) \leq C \, t_0^{-\mu_2}  I^{\sigma_4}\, m(3/4,s)$$
or equivalently 
\be
\label{eqn-linfty70}
\rho_0^{-n} \iint_{Q_{\rho_0,t_0}^* \cap S_{t_0}}  \brw^{2\theta_0 k} \, d\mu \, dt 
\leq C\,  \bar {t_0}^{-\mu_2}  I^{\sigma_2} \, \left (\rho_0^{-n} \iint_{Q_{\rho_0,t_0}^{**} \cap S_{t_0}} \brw \, d\mu \, dt  \right )^{\sigma_3} 
\ee
Since, $\iint_{Q_{\rho_0,t_0}^{**} \cap S_{t_0}} \brw \, d\mu \, dt  \leq C\, I$,  
combining \eqref{eqn-linfty70}  with \eqref{eqn-linfty1} yields  the bound
\be\label{eqn-linfty2}
\|\brw \|_{L^\infty(Q_{\rho_0,t_0}  \cap S_{t_0})} \leq C \, \bar {t_0}^{-\mu}  I^\sigma
\ee
for some new absolute constants $\sigma >0$ and $\mu >0$. The constant $C$ is independent of $r_0$ and $\bar {t_0}$. 
Recalling that $\bar w=\max (w,1)$ we conclude
\eqref{eqn-Hbelow}. 
\end{proof}

Proposition \ref{prop-Hbelow} provides an $L^\infty$ bound on $w(\cdot, t)$ on $M_t \cap  \{ |F| \geq 2 \}$, $0 < t \leq \tau$.
The next result gives an $L^\infty$ bound on $w(\cdot, t)$ on $M_t \cap  \{ |F| \leq 1 \}$. 
For any  given  $t_0 \in (0,\tau]$ and $r \in (0,1)$ 
we consider the parabolic cylinders  in  $\R^{n+1} \times (0,+\infty)$ given by
$$Q_{t_0}:= B_2(0)  \times (t_0/2,t_0]  \qquad \mbox{and}  \qquad Q_{t_0 }^{**}:= B_4(0)  \times (t_0/4,t_0] $$
where $B_r(0):= \{ \, \bx \in \R^{n+1}: \,\, |\bx | < r \}$ denotes the ball in $\R^{n+1}$ centered at the origin of radius $r$. 
We have the  following estimate.

\begin{prop}\label{prop-Hbelow2} Assume that $M_t$ is a solution to  \eqref{eqn-IMCF} as in Theorem
 \ref{thm-ste}   defined for 
$t \in (0,\tau]$ and assume that  $ \tau < T-3\delta$ with $T$ given by \eqref{eqn-T} and  $\delta >0$.
There exist absolute constants $\mu >0$ and $\sigma >0$ and a constant $C$ that depends  on  $\alpha_0, \kappa$,  on $\delta$,
and the initial  bound $\sup_{M_0} \Fom \, H$,    
for which  $w:= \gah(t)  \,(\Fhn \, H)^{-1}$ satisfies the bound
\be\label{eqn-Hbelow2}
\|w\|_{L^\infty( Q_{t_0} \cap S_{t_0})} \leq C \, {t_0}^{-\mu} 
\left (1 + \sup_{t \in (t_0/4, {t_0}]} \int_{M_t \cap  Q^{**}_{t_0}} w(\cdot,t)\, d\mu \right )^{\sigma}
\ee
which holds for any $t_0 >0$ such that $Q_{t_0} \cap S_{t_0}$ is not empty. 
\end{prop}

\begin{proof} The proof is the very similar  as the proof of Proposition  \ref{prop-Hbelow}. 
It is actually simpler as it doesn't need to be scaled with respect to $\rho_0$.   
\end{proof}

\begin{proof}[Proof of Theorem \ref{thm-Hbelow0}]  Readily follows by combining the two estimates in 
Propositions  \ref{prop-Hbelow} and \ref{prop-Hbelow2}. 
\end{proof} 

We will next combine Theorems \ref{thm-lpestw} and \ref{thm-Hbelow0} to obtain the following $L^\infty$ bound on 
$w$  in terms of the  initial data. 

\begin{thm}[$L^\infty$  bound on $w$ in terms of the initial data] \label{thm-linfty}
Assume that $M_t$ is a solution to  \eqref{eqn-IMCF} as in Theorem \ref{thm-ste}  defined for 
$t \in (0,\tau)$, and assume that  $ \tau < T-3\delta$ with $T=T(\alpha_0)$ given by \eqref{eqn-T} and  $\delta >0$.
Then, of any $t_0 \in (0, \tau/2]$ there exists a constant  $C_\delta \big ( t_0, \alpha_0, \kappa, \sup_{M_0} w, 
\inf_{M_0} w \big )$ such that 
\be\label{eqn-Hbelow0} \sup_{ t \in (t_0,\tau) } \|w(\cdot,t)\|_{L^\infty(M_t)} \leq 
C_\delta \big ( t_0, \alpha_0, \kappa, \sup_{M_0} w,  \inf_{M_0} w \big ).  
\ee
\end{thm}

\begin{proof}  We recall the definition of $\bw:= (w-1)_+$,  with $w:=  \gah(t) \, v^{-1} =\gah(t)  \,(\Fhn \, H)^{-1}$ and 
$\hat \gamma$ as defined at the beginning of this section. 
Since $\Fhn = -\Fom \, \omn$ and $-\omn := (\sqrt{1+ |D\bar u|^2})^{-1}$ satisfies $ (\sqrt{1+\alpha_0^2})^{-1} \leq \omn \leq 1$,
it follows that the  assumed  initial bound \eqref{eqn-conHF} and the definition of $w$ imply the bound 
\be\label{eqn-w111} 
 \bar c_0 \leq w (\cdot, 0)  \leq \bar C_0
 \ee
for some positive constants $\bar c_0, \bar C_0$ depending on the constants 
$c_0, C_0$ in \eqref{eqn-conHF} and $\alpha_0, \hat \gamma_0:= \hat \gamma(0)$.

For any $t_0 \in (0,\tau/2]$ we have that \eqref{eqn-linfty} holds. Hence,  it is sufficient to bound
the righthand side of \eqref{eqn-linfty} in terms of the initial data and $t_0$.  Since  $w \leq \hat w +1$  
and $\hat w$ is compactly supported for each $t \in ({t_0}/4, \tau]$ (the latter follows from  \eqref{eqn-limitw} 
and the fact that $\hat \gamma(t) < \gamma(t)$), we have 
\be\label{eqn-tau03}
 \sup_{t \in ({t_0}/4, \tau]} \sup_{R \geq 1}  R^{-n} \int_{M_t \cap \{ |F| \leq R \}} w (\cdot,t)\, d\mu 
\leq 1 + \sup_{t \in ({t_0}/4, \tau]}   \int_{M_t} \hat w (\cdot,t) \, d\mu.
\ee
We next want to apply  the  $L^{p+1}$ bound \eqref{eqn-lte} for $p=0$,  to bound  
$\sup_{t \in ({t_0}/4, \tau]}   \int_{M_t} \hat w (\cdot,t)\, d\mu$ in terms of the initial data and $t_0$. Notice  
that we cannot use \eqref{eqn-lte}  on the interval  $[0,\tau]$, as we have not assumed 
that \eqref{eqn-limitw} holds at $t=0$ which would imply that $\hat w (\cdot, 0)$ is compactly supported. 
It holds only for $t >0$ as a consequence of  parabolic regularity (see Proposition \ref{prop-asb}). 
Thus, we  first apply   \eqref{eqn-lte} on $(t_0/4,\tau]$ to  obtain 
\be\label{eqn-tau02}
\sup_{t \in (t_0/4,\tau]}  \int_{M_t} \bw (\cdot, t) \, d\mu \leq C(\delta,T) \left (1+ \int_{M_{t_0/4}} \bw (\cdot, t_0/4)
\, d\mu \right ).
\ee
To conclude our proof we will   bound   $\int_{M_{t_0/4}} \hat w (\cdot,t_0/4)\, d\mu$ in terms
of $\sup_{M_0 } w$ and the size of the support of $\hat w (\cdot, t_0/4)$.  
Let us first bound $\sup_{M_{t_0/4} } w$ in terms of $\sup_{M_0 } w$. We will do that for $t_0/4 \leq \tau_0$, 
for a  $\tau_0 >0$ depending only on the initial data. This is sufficient since $t_0$ in \eqref{eqn-Hbelow0} can
be chosen small.  To this end,  
 we will use the maximum principle on $w$ to  equation 
\eqref{eqn-eqw2}.  Indeed,  setting  $m(t):= \sup_{M_t} w$, a  straightforward application of the maximum principle 
on equation \eqref{eqn-eqw2},  using also the facts  that $\omn \leq 1$,  $c_1 <0$ and $\gamma^{-1} (t) \leq  \gamma^{-1} (\tau_0)$  
on $[0, \tau_0]$, gives  that 
 $$\frac {d m(t) }{dt} \leq 2 \omn^2 \, \hat \gamma^{-1} (t) \, m(t)^2 + c_1 \, \hat \gamma^{-1} (t) \, m(t)
 \leq 2 \hat \gamma^{-1} (\tau_0) \, m(t)^2$$
 yielding
 $$\sup_{t \in [0, \tau_0] } m(t) \leq \frac{  m(0) \, \hat \gamma (\tau_0)}{ \hat \gamma (\tau_0) - 2 m(0) \, \tau_0}.$$
If  $\tau_0$ is sufficiently small such that $\hat \gamma_0/2 \leq  \hat \gamma(\tau_0) \leq \hat \gamma_0$, 
we conclude that 
$$ \sup_{t \in [0, \tau_0] } m(t)   \leq \frac{  2 m(0) \, \hat  \gamma_0}{ \hat \gamma_0  - 4 m(0)\,  \tau_0}.$$
By decreasing  $\tau_0$ is necessary we may assume that $\hat \gamma_0  - 4 m(0)\,  \tau_0 \geq \hat \gamma_0/2$.
We conclude using also \eqref{eqn-w111} that for such a $\tau_0$ we have 
\be\label{eqn-tau05} 
\sup_{t \in [0, \tau_0] } w (\cdot, t)  \leq 2 \, m(0) \leq 2 \bar C_0.
\ee
Since we may assume without loss of generality that $t_0/4 \leq  \tau_0$,  
the last bound and $\hat w \leq w$ imply   that $\sup_{M_{t_0/4} } \hat w \leq 2 \bar C_0.$
On the other hand,  by  \eqref{eqn-limitw} and the fact that   $\hat \gamma (t) < \gamma (t)$, we have that $\hat w:= (w-1)_+$ 
is compactly supported for all $ t \in (0, \tau)$ and in particular for $t:=t_0/4$. This means that its support is contained in a ball in 
$\R^{n+1}$ of  radius $R_0:=R_0(t_0)$.  Hence, 
\be\label{eqn-tau01}
\int_{M_{{t_0}/4}} \hat w(\cdot, {t_0}/4 ) \, d\mu \leq C  \big ( R_0, \bar C_0 \big ).
\ee
Finally,  by combiinig   \eqref{eqn-linfty}  with  \eqref{eqn-tau03}, \eqref{eqn-tau02}    and \eqref{eqn-tau01} 
we conclude that   \eqref{eqn-Hbelow0} holds.

\end{proof}

\section{Long time existence Theorem \ref{thm-main}}\label{sec-lte}

In this final section we will give the proof of our long time existence Theorem \ref{thm-main}, which says that our  solution $M_t$
of the inverse mean curvature flow will exist up to time $T$, where $T$ denotes the critical time 
where the cone at infinity becomes flat and is given by \eqref{eqn-T}. 
%
%
%

\begin{proof}[Proof of Theorem \ref{thm-main}]  Our short time existence  Theorem \ref{thm-ste}   
implies the existence  of a maximal time $\tau_{\max} >0$ for which a 
convex  solution $M_t$ of 
\eqref{eqn-IMCF}
exists  on $[0,\tau_{\max})$ and the following hold:
\begin{enumerate}[i.]
\item $M_t$, $t \in [0,\tau_{\max})$  is an entire convex graph $x_{n+1} = \uu(x,t)$ over $\R^n$ 
which satisfies condition \eqref{eqn-conet}; 
\item $\uu$ is $C^\infty$ smooth on $\R^n \times (0,\tau_{\max})$;
\item   $ c_{\tau_1} < H \, \Fom  \leq C_0$, on $t \in [0,\tau_1]$, for all $0 < \tau_1 < \tau_{\max}$.
\end{enumerate} 

It $\tau_{\max} =T$, we are done, otherwise $\tau_{\max} < T-\delta$, for some $\delta >0$. 
We claim that  
\be\label{eqn-Hmax1}
\inf_{t \in [0,\tau_{\max})} \inf_{M_t} \, H \, \Fhn   \geq  c_\delta >0. 
\ee
To this end,  we will combine \eqref{eqn-Hbelow0} with \eqref{eqn-tau05}.  We have seen in the proof of Theorem \ref{thm-linfty}
that there exists $\tau_0 >0$ depending only on the initial data such that  \eqref{eqn-tau05} holds. 
Assume without loss of generality  that $\tau_0 < T/2$. 
Since  $w:= \hat \gamma(t) (H\, \Fhn)^{-1}$,  it follows from \eqref{eqn-tau05} that 
$$\inf_{t \in [0,\tau_0]}  \inf_{M_t} \, H \, \Fhn   \geq c_1(\tau_0)  >0.$$ 
Now we apply  \eqref{eqn-Hbelow0} for $t_0:=\tau_0$ and $\tau:=\tau_{\max}$ to obtain the bound
$$\sup_{ t \in (\tau_0,\tau_{\max} ) } \|w(\cdot,t)\|_{L^\infty(M_t)} \leq 
C_\delta \big ( \tau_0 \big ).$$
This can be done since 
conditions i.-iii. above imply that Theorem \ref{thm-linfty} holds on $(0, \tau_{\max})$.
It follows that 
$$\inf_{t \in (\tau_0,\tau_{\max} )}  \inf_{M_t} \, H \, \Fhn   \geq   c_2 (\delta, \tau_0)  >0.$$
Combining the last  two bounds yields that \eqref{eqn-Hmax1} holds and 
since  $\Fhn = - \Fom \, \omn \leq \Fom$ we also have 
\be
\inf_{t \in (0,\tau_{\max} )}  \inf_{M_t} \, H \, \Fom   \geq  c_\delta   >0.
\ee
In addition,   by Proposition \ref{prop-Habove} we have 
\be\label{eqn-Hmax2}
\sup _{M_t} \, H \, \Fom  \leq \sup _{M_t} \, H \, \Fom \leq C_0.
\ee

On the other hand,  $\uu_t \leq 0$ and \eqref{eqn-conet} imply that the pointwise   limit $\uu(x,\tau_{\max}) := \lim_{t \to \tau_{\max}} \uu(x,t)$
exists for all $x \in \R^n$ and it defines a convex graph. Moreover, it satisfies \eqref{eqn-conet} at $t=\tau_{\max}$. 
Now the lower and upper bounds \eqref{eqn-Hmax1}, \eqref{eqn-Hmax2} and   \eqref{eqn-conet} 
 for $t \in [0,\tau_{\max}]$, 
imply that the  fully-nonlinear equation \eqref{eqn-u0} satisfied by $\uu$ is strictly parabolic on compact subsets of $\R^n \times [0,\tau_{\max}]$.
It follows by standard   local regularity results on fully-nonlinear equations  that $\uu(\cdot,\tau_{\max})$  is $C^\infty$ smooth.
Moreover, the above bounds show that $c_\delta \leq H\, \Fom \leq C_0$ on $M_{\tau_{\max}}$. Also, since
$|\Fn|  \leq C_0$ on $M_0$ (see in \eqref{eqn-fn00})   its  evolution equation given  in Lemma \ref{lem-HI1}  and convexity  imply  the bound  
$|\Fn|  \leq C(T)$ on $M_{\tau_{\max}}$. We conclude from the above discussion that at time $t=\tau_{\max}$ 
the entire graph $M_{\tau_{\max}}$ given  by $x_{n+1} = \uu(\cdot, \tau_{\max})$ satisfies all the assumptions  of our short time existence result Theorem \ref{thm-ste},
hence the flow can be extended beyond $\tau_{\max}$ contradicting its maximality.  This shows that $\tau_{\max} = T$,
showing that our solution exists for all $t \in (0,T)$.

\smallskip

Let us now observe that as $t \to T$, the solution converges to a horizontal plane of height $h \in [0,\kappa]$.  First, the pointwise   limit $\uu(x,T):= \lim_{t \to T} \uu(x,t)$
exists,  since $\uu_t(x,t) \geq 0$ for all $t <T$. 
Second, $\uu (\cdot,T)$ is convex and lies between the two horizontal planes $x_{n+1}=0$ and $x_{n+1}=\kappa$. The latter
simply follows from \eqref{eqn-conet} and the fact that $\alpha(T)=0$. In addition, our a'priori local bound from above on the mean curvature $H$   shown in 
Proposition \ref{prop-Hloc-above},  which holds uniformly up to $t=T$,  implies that $\uu (\cdot,T) \in C^{1,1}_{{\text loc}}(\R^n)$. It follows that $\uu (\cdot,T)$ must be
a horizontal plane of height $h \in [0,\kappa]$, and that the convergence  $\lim_{t \to T}  \uu(\cdot,t) \equiv  h$   is in $C^{1,\alpha}$,  on any compact subset 
of $\R^n$ and for all $\alpha < 1$.

\end{proof}

\centerline{\bf Acknowledgements}

\smallskip

\noindent P. Daskalopoulos  has  been partially supported by NSF grant DMS-1600658. She also wishes to thank 
 University of T\"ubingen    and Oberwolfach Research Institute for Mathematics for their hospitality  during the preparation of this work. 

\smallskip

\noindent G. Huisken wishes to thank 
Columbia  University  and  the Institute for Theoretical Studies ITS at ETH Z\"urich  for their hospitality during the preparation of this work.


\begin{thebibliography}{11}


\bibitem{Ca}  Carron, G., {\em In\'egalit\'es de Hardy sur les vari\'et\'es riemanniennes non-compactes}, J. Math. Pures Appl.  {\bf 76} (1997) 883--891.


\bibitem{CD} Choi, K., and Daskalopoulos, P., {\em The $ Q^k$  flow on complete non compact graphs}, arXiv: 1603.03453. 

\bibitem{CDKL}   Choi, K., Daskalopoulos, P.,  Kim, L.  and Lee, L., 
 {\em The evolution of complete non-compact graphs by powers of Gauss curvature}, arXiv: 1603.04286. 


\bibitem{DK} Dahlberg, B.E.J. and Kenig,C.E., {\em Nonnegative solutions to the porous medium equation}, 
Comm. Partial Differential Equations {\bf 9}  (1984), 
409--437.


\bibitem{DD1} Daskalopoulos, P.  and del Pino, M.,   {\em On fast diffusion nonlinear heat equations and a related singular elliptic problem}, Indiana Univ. Math. J. {\bf 43}  (1994) 703--728.

\bibitem{DD2}  Daskalopoulos and P. ; Del Pino, M., {\em   On nonlinear parabolic equations of very fast diffusion},
Arch. Rational Mech. Anal. {\bf 137}  (1997)  363--380. 

\bibitem{DH} Daskalopoulos, P., Huisken, G., King, J.R. {\em Self-similar entire convex solutions to the inverse mean curvature flow}, private communication.  

\bibitem{DKb} Daskalopoulos, P., and Kenig, C., {\em Degenerate diffusions: Initial value problems and local regularity theory},  EMS Tracts in Mathematics, 1,  European Mathematical Society (EMS), Zürich (2007). 


\bibitem{EH1} Ecker, K. and  Huisken, G.,  {\em Mean curvature evolution of entire graphs} , Annals of Math.  {\bf 130} (1989), 453--471. 

\bibitem{EH} {\sc K. Ecker, G. Huisken}, {\em Interior estimates for hypersurfaces moving by mean curvature}, 
Invent. Math.  {\bf 105} (1991), 547--569.

\bibitem{E} {\sc K. Ecker}, {\em Regularity theory for mean curvature flow}, Progress in nonlinear differential 
Equations and their applications, {\bf 57} (2004), BirkhŠuser Boston.

\bibitem{Hal}{\sc H. P. Halldorsson}, {\em Self-similar solutions to the curve shortening flow}, Transactions Amer. Math.Soc.,
{\bf 364}, Number 10, {2012}, 5285--5309.

\bibitem{Ge}{C. Gerhardt}, {\em Flow of nonconvex hypersurfaces into spheres}, J. Diff.Geom. {\bf 32} {1990},
299--314.

\bibitem{Hei}{M. E. Heidusch}, {\em Zur Regularit\"at des inversen mittleren Krümmungsflusses}, Dissertation T\"ubingen Univ. {2001}, 67pp.

\bibitem{HI1}  Huisken, G. and Ilmanen, {\em The Riemannian Penrose Inequality}, Int. Math. Res. Not. 20 (1997) 1045--1058.


\bibitem{HI2}  Huisken, G. and Ilmanen, {\em The Inverse Mean Curvature Flow and the Riemannian Penrose Inequality}, J. Differential Geom. 59 (2001) 353--438.


\bibitem{HI} Huisken, G. and Ilmanen, T., {\em Higher regularity of the inverse mean curvature flow},
 J. Differential Geom  {\bf  80}  (2008)  433--451.

\bibitem{KO}  Kombe, I. and   \"Ozaydin, M. {\em Improved Hardy and Rellich inequalities on Riemannian manifolds},  Trans. Amer. Math. Soc. {\bf 361} (2009) 6191--6203. 

\bibitem{K}{\sc N. V. Krylov}, {\em Nonlinear elliptic and parablic equations of the second order}, Dordrecht:
Reidel, {1987}.

\bibitem{Sm}  Smoczyk, K., {\em Remarks on the inverse mean curvature flow},  Asian J. Math. 4  (2000) 331--335. 

\bibitem{TW} Tian, G. and   Wang, X-J., {\em A priori estimates for fully nonlinear parabolic equations},  Int. Math. Res. Not. IMRN (2013)  3857--3877. 

\bibitem{Ur}{J. Urbas}, {\em On the expansion of strashaped hypersurfaces by symmetric functions of their principal 
curvatures}, Math. Z. {\bf 205} {1990} , 355--372.

\bibitem{Va} V\'azquez, J. L., Nonexistence of solutions for nonlinear heat equations of fast-diffusion type, 
J. Math. Pures Appl., {\bf  71} (1992) 503--526.

\bibitem{Vazquez} V\'azquez, J. L., {\em Smoothing and decay estimates for nonlinear diffusion equations: Equations of porous medium type}  Oxford Lecture Series in Mathematics and its Applications, 33, Oxford University Press, 
(2006)

\end{thebibliography}
\end{document}